\documentclass[a4paper,11pt]{article}
\usepackage{amsfonts}
\usepackage{amsmath}
\usepackage{amsthm}
\usepackage[all]{xy}
\usepackage{amssymb}
\usepackage[english]{babel}
\usepackage{graphicx}
\usepackage{amsmath}
\usepackage{ccaption}
\usepackage{multirow}
\usepackage{latexsym}
\newcommand{\Ann}{\operatorname{Ann}}
\newcommand{\Pic}{\operatorname{Pic}}

\numberwithin{equation}{section}

\begin{document}

\newtheorem{teo}{Theorem}[section]
\newtheorem{lem}[teo]{Lemma}
\newtheorem{pro}[teo]{Proposition}
\newtheorem{cor}[teo]{Corollary}

\theoremstyle{definition}
\newtheorem{defi}[teo]{Definition}
\newtheorem{oss}[teo]{Remark}
\newtheorem{es}[teo]{Example}
\newtheorem{claim}[teo]{Claim}
\newtheorem{dom}[teo]{Questions}
\addcontentsline{}{}{}

\title{Cycles in Jacobians: infinitesimal results}
\author{Emanuele Raviolo}
\maketitle
\begin{abstract}
Let $C$ be a generic smooth curve of genus $g\geqslant 4$. 
We study normal functions and infinitesimal invariants associated to Ceresa cycles $W_{k}-W_{k}^{-}$, $k=2,\cdots,g-2$. 
We show how they can be obtained from the normal function associated to the basic cycle $C-C^{-}$ and, for
$k=2$, we also explicitely determine the zero locus of the infinitesimal invariant. 
For $C$ general hyperelliptic of genus $g=3$, we define the $K-$theoretic counterpart of $W_{2}-W_{2}^{-}$, 
generalizing a construction of A. Collino, and show that it is 
indecomposable. 
\end{abstract}
\section*{Introduction}
Let $C$ be a smooth projective curve of genus $g\geqslant3$. 
The $k-$fold symmetric product $C_{k}$ maps into the Jacobian $J(C)$ via the Abel-Jacobi map.
Let us call the image $W_{k}$, which is birational to $C_{k}$ for $k=1,\dots,g-2$.
We can consider the cycle-theoretic push-forward of $W_{k}$ under the involution $-1:J(C)\to J(C)$ of the Jacobian: $W_{k}^{-}=(-1)_{*}W_{k}$.
The difference $W_{k}-W_{k}^{-}$ is an homologically trivial cycle, since $(-1)_{*}$ acts as the identity on even homology.
It is natural to ask whether these cycles, called $\textit{Ceresa cycles}$, are algebraically trivial.
The answer is given by Ceresa theorem.
\begin{teo}[\cite{cer}] 
\label{c}   
If $C$ is general then $W_{k}-W_{k}^{-}$ is not algebraically trivial for $k=1,\dots,g$.
\end{teo}
These cycles appear in different contexts.
From the point of view of algebraic cycles, an important consequence of Theorem \ref{c} is the construction of A. Beauville \cite{bea} of a subring, 
called \textit{tautological}, of the Chow ring modulo algebraic equivalence of the generic Jacobian, which has been extensively studied by Beauville himself 
and other authors.
We also mention the results of R. Hain \cite{H} relating the cycle $C-C^{-}$ to the geometry of the mapping class group.

 The original proof of Ceresa was by degeneration.
There are now many different proofs of Ceresa theorem in the case $k=1$, i.e. for the \textit{basic cycle} $C-C^{-}$.
One particularly interesting is the one given by A. Collino and G.P. Pirola \cite{cp}.
They studied the normal function $\nu$ associated to the basic cycle and computed its infinitesimal invariant $\delta\nu$.
From the non triviality of $\delta\nu$, they were able to deduce that $C-C^{-}$ is not algebraically trivial.
They also proved that, if $g=3$, then the infinitesimal invariant $\delta\nu$ gives the equation of the canonical curve $C\subset\mathbb{P}^{2}$. 
We also remind the proof of B. Harris \cite{har} which uses harmonic volumes.
His work was then completed by W. Faucette \cite{f}, who extended Harris' results to all Ceresa cycles.

We present here some results on the normal functions $\nu_{k}$ and infinitesimal invariants $\delta\nu_{k}$ 
associated to $W_{k}-W_{k}^{-}$, $k\geqslant2$, for a curve of genus $g\geqslant4$.
Our study of the normal functions $\nu_{k}$ clarifies what is the relation between $\nu$ and $\nu_{k}$.
To be more concrete suppose that $\mathcal{C}\to B$ is a family of smooth curves of genus $g$. 
If we denote by $\mathcal{J}^{g-k}\to B$ the family of intermediate Jacobians we get normal functions $\nu_{k}\in H^{0}(B,\mathcal{J}^{g-k})$, with $\nu_{1}=\nu$ 
(see Section \ref{ht} for the basic definitions).
Then we prove
\begin{teo} \label{1}
\begin{enumerate}
  \item $\theta^{k-1}_{J(C_{t})}\nu_{k}(t)=C(g,k)\nu_{1}(t)$, where $\theta_{J(C_{t})}$ is the class of the divisor of $J(C_{t})$ and $C(g,k)$ is a rational 
  constant depending only on $g$ and $k$.
  \item According to the Lefschetz decomposition on cohomology, the normal functions decomposes as $\nu_{k}=\nu_{k}^{(0)}+\cdots+\nu_{k}^{(k)}$.
  Then, if $g\geqslant5$, $\nu_{k}^{(i)}=0$ for $i=0,\dots,k-2$.
\item $\delta\nu_{k}(t)\circ(\theta^{k-1}\otimes\textrm{Id})=C(g,k)\delta\nu_{1}(t)$.
In particular $W_{k}-W_{k}^{-}$ is not algebraically trivial for $k=2,\dots,g-2$.
\end{enumerate}
\end{teo}
We then specialize to the normal function $\nu_{2}$.
We are able to determine the zero locus of a section of a certain vector bundle naturally attached to the infinitesimal invariant.
Instead of using the formula in the above theorem, we use a machinery due to G.P. Pirola and F. Zucconi \cite{pz}.
Their work enable us to compute the infinitesimal invariant $\delta\nu_{2}(t)$ on some decomposable tensors in $H^{3,2}(J(C_{t}))\otimes H^{1}(C_{t},T_{C_{t}})$.
\\
To understand the geometry behind $\delta\nu_{2}(0)$, let us identify the curve with its canonical image in the canonical space $\mathbb{P}^{g-1}$.
consider the Grassmanian \mbox{$\mathbb{G}(g-4,\mathbb{P}^{g-1})$.}
We consider the subscheme $\mathcal{D}$ of the Grassmanian $\mathbb{G}(g-4,\mathbb{P}^{g-1})$ made of the linear $(g-4)$-spaces intersecting $C$.
The infinitesimal variations of Hodge structures of the second symmetric product $C_{2}$ produce a variety $\Gamma$
contained in $\mathbb{P}H^{1}(C,T_{C})\times\mathbb{G}(g-4,\mathbb{P}^{g-1})$.
The image of $\Gamma$ under the second projection $\pi_{2}$ contains $\mathcal{D}$. 
Moreover the infinitesimal invariant $\delta\nu_{2}(0)$ defines in a natural way a section $\alpha$ of a vector bundle on $\Gamma$.
Let us call $Z$ its zero locus.
Then we prove the following. 
\begin{teo}\label{2}  
$\pi_{2}(Z)=\mathcal{D}$.
In particular the infinitesimal invariant $\delta\nu_{2}(0)$ reconstructs the curve $C$. 
\end{teo}

We then consider an hyperelliptic curve $C$ of genus $g$. 
It is well-known that $W_{k}^{-}$ is a translate of $W_{k}$ (because the hyperelliptic involution induces $-1$ on $J(C)$) and then $W_{k}-W_{k}^{-}$
is algebraically trivial.

Nevertheless, A. Collino \cite{c} constructed a cycle $Z(C)$ in the higher Chow group $CH^{g}(J(C),1)$ and proved that it is indecomposable. 
As Collino explains in his paper, this cycle is a sort of substitute of $C-C^{-}$ and in fact the behaviour of these two cycles,
 from an infinitesimal point of view, is quite similar.

If $X$ is a smooth projective variety, it is considerably hard to construct indecomposable higher Chow cycles in $CH^{p}(X,1)$ for $p<\dim(X)$. 
Motivated by the above results, we are able to provide examples of indecomposable higher Chow cycles in $CH^{2}(J(C),1)$ when $g=3$. 
We consider, in fact, certain cycles $Z_{2}(C)\in CH^{2}(J(C),1)$ defined using the symmetric product $C_{2}$.
This is the natural substitutes of $W_{2}-W_{2}^{-}$.
As in the case of Ceresa cycles, we relate the associated infinitesimal invariant to that of Collino cycle and prove
\begin{teo}\label{3}
The cycle $Z_{2}(C)$ is indecomposable for the generic hyperelliptic curve of \mbox{genus $3$.}
\end{teo}

The paper is organized as follows. 
The first section contains a collection of two basic notions from Hodge theory: normal functions and infinitesimal invariants.
In section \ref{ceresa} we prove our results on the normal functions $\nu_{k}$. 
Section \ref{i.i.} is devoted to the study of the zero locus of $\delta\nu_{2}$.
In the last section ,we briefly introduce higher Chow groups, the regulator map, and then we prove Theorem \ref{3}.
\section{Basic notions of Hodge theory}
\subsection{Normal functions and infinitesimal invariants}
\label{ht}
We recall some basic facts from Hodge theory. The reader is referred to \cite[ch. 7]{voi2} for details.
\\
Let $\pi:\mathcal{X}\rightarrow B$ be a smooth projective morphism between smooth complex manifolds. There are holomorphic
 vector bundles $\mathcal{H}^{k}$ over $B$ with fibers $\mathcal{H}^{k}_{t}=H^{k}(X_{t},\mathbb{C})$. For fixed $k$ we also have holomorphic subbundles
 $\mathcal{F}^{i}\mathcal{H}^{k}\subset\mathcal{H}^{k}$ with fiber $(\mathcal{F}^{i}\mathcal{H}^{k})_{t}=F^{i}H^{k}(X_{t},\mathbb{C})$, $F$ being the Hodge filtration on the cohomology groups. We also recall that the Gauss-Manin connection $\nabla:\mathcal{H}^{k}\rightarrow\mathcal{H}^{k}\otimes\Omega^{1}_{B}$ 
 is flat, i.e $\nabla^{2}=0$, and moreover  
 \begin{equation}
\label{gm}
\nabla(\mathcal{F}^{i}\mathcal{H}^{k})\subset\mathcal{F}^{i-1}\mathcal{H}^{k}\otimes\Omega^{1}_{B}
\end{equation}
 Setting $\mathcal{J}^{p}(\mathcal{X})=\frac{\displaystyle\mathcal{H}^{2p-1}}{\displaystyle\mathcal{F}^{p}\mathcal{H}^{2p-1}+R^{2p-1}\pi_{*}\mathbb{Z}}$ we get a fibration 
 in complex tori $j(\pi):\mathcal{J}^{p}(\mathcal{X})\rightarrow B$, with $\mathcal{J}^{p}_{t}=J^{p}(X_{t})$ the \textit{p-th intermediate jacobian} of $X_{t}$.
 
 Let us denote by $\mathcal{Z}^{p}_{hom}(\mathcal{X}/B)$ the group of codimension $p$ cycles $\mathcal{Z}\subset\mathcal{X}$ 
 with support flat over $B$ and such that $Z_{b}=\mathcal{Z}\cdot X_{b}$ is homologous to zero in $X_{b}$. 
 An element $\mathcal{Z}\in\mathcal{Z}^{p}_{hom}(\mathcal{X}/B)$ induces a section $H^{0}(B,\mathcal{J}^{p})$ by the rule $\nu(t)=AJ_{X_{t}}(Z_{t})$.
 It is well-known \cite[prop. 7.19]{voi2} that $\nu$ satisfies the \textit{transversality property}: 
for every holomorphic local lifting $\tilde{\nu}:U\rightarrow\mathcal{H}^{2p-1}$, $\nabla\tilde{\nu}\in\mathcal{F}^{p-1}\mathcal{H}^{2p-1}\otimes\Omega^{1}_{U}$.
This means that $\nu$ is a \textit{normal function}.
 
 Thanks to (\ref{gm}) the Gauss-Manin induces
 \[
 \nabla^{p-1}_{p}: \mathcal{H}^{p,p-1}\rightarrow\mathcal{H}^{p-1,p}\otimes\Omega^{1}_{B}
 \]
where $\mathcal{H}^{p,p-1}=\frac{\displaystyle\mathcal{F}^{p}\mathcal{H}^{2p-1}}{\displaystyle\mathcal{F}^{p+1}\mathcal{H}^{2p-1}}$ and
$\mathcal{H}^{p-1,p}=\frac{\displaystyle\mathcal{F}^{p-1}\mathcal{H}^{2p-1}}{\displaystyle\mathcal{F}^{p}\mathcal{H}^{2p-1}}$.

The class $[\nabla\tilde{\nu}]\in\textrm{coker}\nabla_{p}^{p-1}$ is independent of the chosen lifting. 
It is called the \textit{infinitesimal invariant} of $\nu$ and it is denoted by $\delta\nu$ \cite{g}. 
Under the isomorphism
\begin{equation}
\label{dual}
\textrm{coker}\nabla_{p}^{p-1}\cong\big(\ker\big((\nabla_{p}^{p-1})^{t}:\mathcal{H}^{n-p+1,n-p}\otimes T_{B}\rightarrow\mathcal{H}^{n-p,n-p+1}\big)\big)^{*}
\end{equation}
 we can then consider $\delta\nu$ as an element of the last vector bundle. 
 More precisely, if $\sum_{i}\Omega_{i}\otimes v_{i}\in\ker\nabla_{p}^{p-1}(0)^{t}$, 
 we have 
 \begin{equation}
 \label{def}
 \delta\nu(0)(\sum_{i}\Omega_{i}\otimes v_{i})=\sum_{i}\langle\nabla_{v_{i}}\tilde{\nu},\Omega_{i}\rangle=\sum_{i}\int_{X}\nabla_{v_{i}}\widetilde{\nu}\wedge\Omega_{i}
 \end{equation}
 (cf. \cite{grif}).
\subsection{Adjoint class}\label{adjunction}
Let $X$ be a smooth projective $n$-dimensional variety and $\mathcal{F}$ a locally free sheaf of rank $n$.
Take a non-zero element $\xi\in Ext^{1}(\mathcal{F},\mathcal{O}_{X})$. It is the class of a short exact sequence 
\[
0\rightarrow\mathcal{O}_{X}\rightarrow\mathcal{E}\rightarrow\mathcal{F}\rightarrow0. 
\]
By abuse of notation we also call $\xi$ the element of $H^{1}(X,\mathcal{F}^{*})$ corresponding to the extension class 
under the isomorphism $ Ext^{1}(\mathcal{F},\mathcal{O}_{X})\cong H^{1}(X,\mathcal{F}^{*})$.
Then the coboundary map $\partial_{\xi}$ in the above sequence is given by cupping with $\xi$:
\[
\begin{array}{lcll}
\partial_{\xi}: & H^{0}(X,\mathcal{F}) & \rightarrow  & H^{1}(X,\mathcal{O}_{X})      \\
& \omega & \mapsto & \xi\cdot\omega        
\end{array}.
\]
To simplify notations, we will denote the homomorphism $\partial_{\xi}$ simply by $\xi$.

Let us assume now that $\dim(\ker\xi)\geqslant n+1$, and that $W\subset\ker\xi$ has dimension $n+1$. 
We fix a basis $\mathcal{B}=\{\omega_{1},\dots,\omega_{n+1}\}$ of $W$.
Then there exist liftings $\widetilde{\omega}_{i}\in H^{0}(X,\mathcal{E})$ of $\omega_{i}$ for $i=1,\dots,n+1$. 
Let us consider the natural map
\[
\Lambda:\bigwedge^{n+1}H^{0}(X,\mathcal{E})\rightarrow H^{0}(X,\bigwedge^{n+1}\mathcal{E})\cong H^{0}(X,\bigwedge^{n}\mathcal{F}).
\]
We define
\[
\omega_{\xi,W,\mathcal{B}}=\Lambda(\widetilde{\omega}_{1}\wedge\cdots\wedge \widetilde{\omega}_{n+1})\in H^{0}(X,\bigwedge^{n}\mathcal{F}).
\]
The form $\omega_{\xi,W,\mathcal{B}}$ depends on two choices: the base $\mathcal{B}$ and the liftings $\widetilde{\omega}_{1},\dots,\widetilde{\omega}_{n+1}$.
If we take another basis $\mathcal{B}^{\prime}$ of $W$, then $\omega_{\xi,W,\mathcal{B}}=\lambda\omega_{\xi,W,\mathcal{B}^{\prime}}$, 
where $\lambda$ is the determinant of the base change matrix.
If we change instead the liftings $\widetilde{\omega}_{1},\dots,\widetilde{\omega}_{n+1}$ then one can check that $\omega_{\xi,W,\mathcal{B}}$ 
changes by an element of $W^{n}:=\Lambda(W)$. 

The class $[\omega_{\xi,W,\mathcal{B}}]\in H^{0}(X,\bigwedge^{n}\mathcal{F})/W^{n}$ is then well-defined.
\begin{defi}[\cite{pz}]
The element $[\omega_{\xi,W,\mathcal{B}}]\in H^{0}(X,\bigwedge^{n}\mathcal{F})/W^{n}$ is called \textit{adjoint class}. 
\\
When $\mathcal{F}=\Omega^{1}_{X}$ we will denote by $\omega_{\xi,W,\mathcal{B}}$ the representative which is orthogonal to $W^{n}$ and call
it \textit{adjoint form}.
\end{defi}
With a little abuse of notation we will often abbreviate $\omega_{\xi,W,\mathcal{B}}$ with $\omega_{\xi,W}$.
Notice also that $[\omega_{\xi,W,\mathcal{B}}]=0$ if and only if $\omega_{\xi,W,\mathcal{B}}=0$.
 
 From now on we specialize to the case $\mathcal{F}=\Omega^{1}_{X}$. 
 Then $Ext^{1}(\mathcal{F},\mathcal{O}_{X})\cong H^{1}(X,T_{X})$ and if
 $\xi\in H^{1}(X,T_{X})$ then it represents a first order deformation $\mathcal{X}\rightarrow\textrm{Spec}\frac{\displaystyle\mathbb{C}[t]}{\displaystyle(t^{2})}$.
 The extension class has the following form
\begin{equation}
\label{ext}
0\rightarrow\mathcal{O}_{X}\rightarrow\Omega^{1}_{\mathcal{X}|X}\rightarrow\Omega^{1}_{X}\rightarrow0.
\end{equation}
The following theorem collects some of the results in \cite{pz}.
 \begin{teo}[\cite{pz}]
 \label{imp}
 Suppose that $W^{n}\neq0$ and denote by $D$ and $Z$ the fixed and the moving part of the linear system $|W^{n}|$. Then:
 \begin{enumerate}
  \item $[\omega_{\xi,W,\mathcal{B}}]\in H^{0}(X,\omega_{X}(-D)\otimes\mathcal{I}_{Z})$, where $\mathcal{I}_{Z}$ is the ideal sheaf of the scheme $Z$.
  \item If $[\omega_{\xi,W,\mathcal{B}}]=0$ then $\xi$ is in the kernel of $H^{1}(X,T_{X})\rightarrow H^{1}(X,T_{X}(D))$. 
  In particular if $D=\emptyset$ then $[\omega_{\xi,W,\mathcal{B}}]\neq0$.
\end{enumerate}
 \end{teo} 
\begin{oss}
\label{curves}
Suppose that in the above setting $X$ is a smooth projective curve.
Then by \cite[Theorem 1.1.8. p. 62]{cp} the converse of the second assertion hold: if $\xi$ is in the kernel of $H^{1}(X,T_{X})\rightarrow H^{1}(X,T_{X}(D))$
then $[\omega_{\xi,W,\mathcal{B}}]=0$.
\end{oss} 
We want to give an expression of the adjoint form in local coordinates which will be useful later. 
Let us restrict ourselves to the case  of interest to us, $\mathcal{F}=\Omega^{1}_{X}$.
Let $\underline{z}=(z_{1},\dots,z_{n})$ be local coordinates on $X$.
If we write locally $\omega_{i}=\sum_{j=1}^{n}f_{ij}(\underline{z})dz_{j}$, then the local expression of the lifting is 
$\widetilde{\omega}_{i}=\sum_{j=1}^{n}f_{ij}(\underline{z})dz_{j}+g_{i}(\underline{z})dt$.

Then, locally, 
\[
\widetilde{\omega}_{1}\wedge\cdots\wedge\widetilde{\omega}_{n+1}=\det
\begin{pmatrix}
f_{11} & \cdots & f_{1n} & g_{1} \\
\vdots & \vdots & \vdots & \vdots \\
f_{n+1n} & \cdots & f_{n+1n} & g_{n}
\end{pmatrix}
dz_{1}\wedge\cdots\wedge dz_{n}\wedge dt
.\]
Now since the isomorphism $L_{\xi}:H^{0}(X,\bigwedge^{n+1}\mathcal{E})\cong H^{0}(X,\omega_{X})$ is the contration with $\frac{\partial}{\partial t}$ 
we get that, locally, 
a representative for $[\omega_{\xi,W,\mathcal{B}}]$ is given by
\[
L_{\xi}(\widetilde{\omega}_{1}\wedge\cdots\wedge\widetilde{\omega}_{n+1})=\det
\begin{pmatrix}
f_{11} & \cdots & f_{1n} & g_{1} \\
\vdots & \vdots & \vdots & \vdots \\
f_{n+1n} & \cdots & f_{n+1n} & g_{n}
\end{pmatrix}
dz_{1}\wedge\cdots\wedge dz_{n}
.\]

We now consider the case we will deal with in the next section of the paper. 
Namely, let $C$ be a smooth non-hyperelliptic curve of genus $g\geqslant 4$ and consider the Ceresa cycles $W_{k}-W_{k}^{-}$, $k=1,\dots,g-2$.
For $k=1$ we will use the notation $C-C^{-}$ and call it the \textit{basic cycle}.
We consider a family $\pi:\mathcal{C}\rightarrow B$  with $\pi^{-1}(0)=C$. 
Therefore we have the corresponding family of Jacobians $\mathcal{J}\rightarrow B$ 
 and also relative Ceresa cycles $\mathcal{W}_{k}-\mathcal{W}_{k}^{-}\subset\mathcal{J}$ 
 (where $\mathcal{J}=\mathcal{J}^{1}(\mathcal{C})$ with the notation of \ref{ht}). 
 The normal function $\nu_{k}:B\rightarrow\mathcal{J}^{g-k}(\mathcal{J})$ is 
 given by the Abel-Jacobi map as explained in \ref{ht}.
\\
We recall a formula proved in \cite{pz} which computes the infinitesimal invariant in terms of the adjoint form.
Although their formula is quite general, we restrict ourselves to the case of interest for us, namely the infinitesimal invariant of $\nu_{k}$.
We assume below that the family $\mathcal{C}\to B$ is the Kuranishi family of $C$; in particular $T_{B,0}\cong H^{1}(C,T_{C})$.
\begin{teo}[\cite{pz}]
\label{adj}
Let $W=\langle\Omega_{1},\dots,\Omega_{k+1}\rangle\subset H^{1,0}(C_{k})$ be a $(k+1)$-dimensional space and let $\xi\in H^{1}(C_{k},T_{C_{k}})$
be a non zero element such that $\xi\cdot\Omega_{i}=0$ for every $i$. 
Then, for every $\sigma\in H^{k,0}(J(C))$ we have
  \[
  \delta\nu_{k}(0)(\Omega_{1}\wedge\cdots\wedge\Omega_{k+1}\wedge\bar{\sigma}\otimes \xi)=2\int_{C_{k}}\omega_{\xi,W}\wedge\overline{\sigma}.
  \]
\end{teo}
\begin{oss}\label{basic}
1) In the above formula we are using the isomorphisms $H^{0}(C_{k},\omega_{C_{k}})\cong\bigwedge^{k}H^{0}(C,\omega_{C})$ \cite{m61} and 
$H^{0}(C,\omega_{C})\cong H^{1,0}(J(C))$.
\\
2) The case $k=1$, i.e. the infinitesimal invariant of the basic cycle $C-C^{-}$, was computed in \cite{cp}.
\end{oss}

\section{Normal functions for $W_{k}-W_{k}^{-}$}\label{ceresa}
\subsection{Abel-Jacobi map}
In this section we compare the Abel-Jacobi image of $W_{k}-W_{k}^{-}$ with that of $C-C^{-}$, for a non-hyperelliptic curve of genus $g\geqslant 4$.
Recall that there are intermediate Jacobians
\[J^{g-k}(J(C))=\frac{\displaystyle{H^{2g-2k-1}(J(C))}}{\displaystyle{F^{g-k}H^{2g-2k-1}(J(C))}+H^{2g-2k-1}(J(C),\mathbb{Z})}\]
 for every $k$.
Serre and Poincar\'e dualities induce isomorphisms $J^{g-1}(J(C))\cong\frac{\displaystyle{(F^{2}H^{3}(J(C)))^{*}}}{\displaystyle{H_{3}(J(C),\mathbb{Z})}}$ and
$J^{g-k}(J(C))\cong\frac{\displaystyle{(F^{k+1}H^{2k+1}(J(C)))^{*}}}{\displaystyle{H_{2k+1}(J(C),\mathbb{Z})}}$. 
Recall also that there are Abel-Jacobi mappings $AJ_{k}:\mathcal{Z}_{k}^{hom}(J(C))\to J^{g-k}(J(C))$ defined by 
$AJ_{k}(Z)(\omega)=\int_{\sigma}\omega$, 
where $\sigma$ is a $(2k+1)-$chain such that $\partial \sigma=Z$.
\\
Notice also that thanks to the Hard Lefschetz theorem and Lefschetz decomposition we have isomorphisms
\begin{equation}
\label{lefschetz}
H^{2g-2k-1}(J(C),\mathbb{C})\cong\theta^{g-2k-1}H^{2k+1}(J(C),\mathbb{C})\cong\bigoplus_{i=0}^{k}\theta^{g-2k+i-1}P^{2k-2i+1}(J(C)),
\end{equation}
$P^{j}(J(C))=\ker(\theta^{g-j+1}:H^{j}(J(C),\mathbb{C})\rightarrow H^{2g-j+2}(J(C),\mathbb{C}))$ being the primitive cohomology.
According to the above decomposition the Abel-Jacobi map has several components: $AJ_{k}=\sum_{i=0}^{k}AJ^{(i)}_{k}$.
\begin{oss}
The Abel-Jacobi map $u_{k}:C_{k}\to J(C)$ depends on the choice of a point $p_{0}\in C$ so that, a priori, also the Abel-Jacobi image 
$AJ_{k}(W_{k}-W_{k}^{-})$ depens on it.
However, reasoning as in \cite[p. 72]{cp}, one can show that $AJ_{k}^{(i)}(W_{k}-W_{k}^{-})$ does not
 depend on $p_{0}$ for $i\neq k$. 
 \end{oss}
We want to compare the Abel-Jacobi images of $C-C^{-}\in\mathcal{Z}_{1}^{hom}(J(C))$ and $W_{k}-W_{k}^{-}\in\mathcal{Z}_{k}^{hom}(J(C))$.
In what follows we use the basic properties of the Pontryagin product; see for example \cite[pp. 21-22]{bl}.

Let us denote by $\Delta$ a $3-$ chain such that $\partial\Delta=C-C^{-}$ and by $\Delta_{k}$  a $(2k+1)-$chain such that 
$\partial\Delta_{k}=W_{k}-W_{k}^{-}$. 
\\
Define also the cycle 
\[
\Gamma_{k}=\sum_{i=0}^{k-1}C^{*i}*(C^{*(k-i-1)})^{-}.
\]
By the properties of the Pontryagin product \cite[pp. 21-22]{bl} and the fact that $C^{*k}=k!W_{k}$ 
for every $k$ we find
$(C-C^{-})*\Gamma_{k}=k!(W_{k}-W_{k}^{-})$.
   Consider the following diagram 
\[
\xymatrix{
\mathcal{Z}_{1}^{hom}(J(C)) \ar[d]^{*\Gamma_{k}} \ar[r]^{AJ_{1}}&  J^{g-1}(J(C))\\
\mathcal{Z}_{k}^{hom}(J(C)) \ar[r]^{AJ_{k}}     & J^{g-k}(J(C))\ar[u]_{\theta^{k-1}}   
,}
\]
where $*$ denotes the Pontryagin product \cite[p. 21]{bl} (which defined at the level chains) and $\theta\in H^{2}(J(C),\mathbb{Z})$ 
is the cohomology class of the theta divisor. 
With the above notations we have
\begin{equation}\label{abel}
\begin{array}{cc}
    AJ_{1}(C-C^{-})= & \displaystyle\int_{\Delta} \vspace{0.1cm}  \\
    AJ_{k}(W_{k}-W_{k}^{-})= & \displaystyle\int_{\Delta_{k}} \vspace{0.1cm} \\
    \theta^{k-1} AJ_{k}(W_{k}-W_{k}^{-})=& \displaystyle\int_{\Delta_{k}}\theta^{k-1}\wedge  
\end{array}.
\end{equation}
\begin{teo}\label{comparison}
We have
\[
\theta^{k-1} AJ_{k}(W_{k}-W_{k}^{-})=C(g,k)AJ_{1}(C-C^{-}),
\]
where $C(g,k)$ is a rational constant depending only on $g$ and $k$.
\end{teo}
\begin{proof}
By relations (\ref{abel}) the result will follow if we prove that
\begin{equation}
\label{thesis}
\displaystyle\int_{\Delta_{k}}\theta^{k-1}\wedge\sigma=C(g,k)\displaystyle\int_{\Delta}\sigma\textrm{ for every }\sigma\in F^{2}H^{3}(J(C)).
\end{equation}
Let us consider the $(2k+1)$-chain $\Delta*\Gamma_{k}$.
By the properties of the Pontryagin product \cite[pp. 21-22]{bl} its boundary is
\[
\partial\Delta_{k}=\partial(\Delta*\Gamma_{k})=(C-C^{-})*\Gamma_{k}=k!(W_{k}-W_{k}^{-})
\]
and then
\begin{equation}\label{pont}
\displaystyle\int_{\Delta_{k}}\theta^{k-1}\wedge\sigma=\frac{1}{k!}\displaystyle\int_{\Delta*\Gamma_{k}}\theta^{k-1}\wedge\sigma 
\end{equation}
The result is then equivalent to the equality
\begin{equation}
\label{tesi}
\frac{1}{k!}\displaystyle\int_{\Delta*\Gamma_{k}}\theta^{k-1}\wedge\sigma=C(g,k)\displaystyle\int_{\Delta}\sigma\textrm{ for every }\sigma\in F^{2}H^{3}(J(C)).
\end{equation}
We first give the details of the proof of (\ref{tesi}) in the case $k=2$.
Let us denote by $z_{1},\dots,z_{g}$ complex coordinates on $J(C)$. Then $F^{2}H^{3}(J(C))$ is generated by the differentials
$\{dz_{i}dz_{j}d\overline{z}_{k},dz_{l}dz_{m}dz_{n}\}$ (we omit the wedge symbol).
We can suppose that $\sigma=dz_{1}dz_{2}d\overline{z}_{3}$, the proof for the other differentials being similar.
With this coordinates $\theta=\frac{i}{2}\sum_{i=1}^{g}dz_{i}\wedge d\overline{z}_{i}$.
If we denote by 
\[
\mu:J(C)\times J(C)\rightarrow J(C)
\] 
the summation map we have $\mu^{*}dz_{i}=du_{i}+dv_{i}$,
where $u_{1},\dots,u_{g}$ and $v_{1},\dots,v_{g}$ are coordinates on the first and second factor respectively.
\\
We have
$\mu^{*}\theta=p_{1}^{*}\theta+p^{*}_{2}\theta+\psi$, where $\psi=\sum_{i}(du_{i}\wedge d\overline{v}_{i}+d\overline{u}_{i}\wedge dv_{i})$
and $p_{1}$, $p_{2}$ the two projections. 
Since $\Delta*(C+C^{-})=\mu_{*}(\Delta\times (C+C^{-}))$ we have
\[
\displaystyle\int_{\Delta*(C+C^{-})}\theta\wedge\sigma=\displaystyle\int_{\Delta\times (C+C^{-})}\mu^{*}\theta\wedge\mu^{*}\sigma
.\]
We immediately see that $\displaystyle\int_{\Delta\times(C+C^{-})}p_{1}^{*}\theta\wedge\mu^{*}\sigma=0$.
To compute the other terms just recall that the homology class of $C$ is 
$  [C]=-2i\sum_{i=1}^{g}a_{i}\wedge \overline{a}_{i}\in H_{2}(J(C),\mathbb{C})$,
  where $a_{i}\in H_{1}(J(C),\mathbb{C})$ are Poincar\'e duals of $dv_{i}$ and then $\displaystyle\int_{C} dv_{i}d\overline{v}_{j}=\displaystyle\int_{C^{-}}dv_{i}d\overline{v}_{j}=\delta_{i,j}$.
 After a straightforward computation we get 
  $ \frac{1}{2}\displaystyle\int_{\Delta\times(C+C^{-})}p_{2}^{*}\theta\wedge\mu^{*}\sigma=     
     g\displaystyle\int_{\Delta}\sigma$,
    $\frac{1}{2}\displaystyle\int_{\Delta\times(C+C^{-})}\psi\wedge\mu^{*}\sigma= -3\displaystyle\int_{\Delta}\sigma.$

Patching together we find
$\frac{1}{2}\displaystyle\int_{\Delta*(C+C^{-})}\theta\wedge\sigma=(g-3)\displaystyle\int_{\Delta}\sigma$
which, according to (\ref{thesis}) and (\ref{pont}) gives the result.
\\
To prove the result in the general case $k\geqslant2$ notice that 
 $[C^{*i}*(C^{*(k-i-1)})]=(k-1)![W_{k-1}]$ for \mbox{every $i$.} 
 Then the proof of (\ref{tesi}) proceeds as in the case $k=2$. 
\end{proof}
\begin{pro}\label{0}
We have $AJ_{k}^{(i)}(W_{k}-W_{k}^{-})=0$ for $i=0,\dots,k-2$.
\end{pro}
\begin{proof}
Notice that the Pontryagin product in homology corresponds to the wedge product in cohomology under Poincar\`e duality.
Under the isomorphisms $H^{2g-3}(J(C),\mathbb{C})\cong H_{3}(J(C),\mathbb{C})$ and $H^{2g-2k-1}(J(C),\mathbb{C})\cong H_{2k+1}(J(C),\mathbb{C})$
we have $AJ_{1}(C-C^{-})=\Delta$ and $k!AJ_{k}(W_{k}-W_{k}^{-})=\Delta*\Gamma_{k}=AJ_{1}(C-C^{-})*\Gamma_{k}$ by the discussion 
at the beginning of this section.
The chain $\Gamma_{k}$ is homologous to $C^{*(k-1)}$ and the map $[C^{*(k-1)}]*:H_{3}(J(C),\mathbb{C})\to H_{2k+1}(J(C),\mathbb{C})$
is injective, hence $AJ_{k}(W_{k}-W_{k}^{-})\in [C^{*(k-1)}]*H_{3}(J(C),\mathbb{C})$.
Using again Poincar\`e duality we obtain the result.
\end{proof}

\subsection{Normal functions in genus $g\geqslant 5$}
Consider a generic curve $C$ of genus $g\geqslant 5$ and a family $\pi:\mathcal{C}\rightarrow B$ with central fiber $C$ . 
Let us denote by $\mathcal{P}^{j}$ the bundle over $B$ with fiber 
$\mathcal{P}^{j}_{t}=P^{j}(J(C_{t})):=\ker(\theta_{J(C_{t})}^{g-j+1}:H^{j}(J(C_{t}),\mathbb{C})\rightarrow H^{2g-j+2}(J(C_{t}),\mathbb{C}))$.
Hard Lefschetz theorem and Lefschetz decomposition induces a splitting of the family of intermediate jacobians
\begin{equation}
\label{global}
     \mathcal{J}^{g-k}=\mathbf{\Theta}^{g-2k-1}\mathcal{P}^{2k+1}\oplus\cdots\oplus\mathbf{\Theta}^{g-k-3}\mathcal{P}^{5}
     \oplus\mathbf{\Theta}^{g-k-2}\mathcal{P}^{3} \oplus\mathbf{\Theta}^{g-k-1}\mathcal{J}.
     \end{equation}
where $\mathbf{ \Theta}$ is the section of $R^{2}j(\pi)_{*}\mathbb{Z}$ given by $\mathbf{ \Theta}(t)=\theta_{J(C_{t})}$. 
\\
We call $\nu_{k}^{(i)}$ the component of $\nu_{k}$ in $\mathbf{\Theta}^{g-2k+i-1}\mathcal{P}^{2k-2i-1}$ with respect to the above decomposition 
for $i=0,\dots,k-2$.
The results of section \ref{ceresa} can be translated in the language of normal functions:
\begin{enumerate}
  \item $\nu_{k}=C(k,g)\mathbf{\Theta}^{k-1}\nu_{1}$
  \item $\nu_{k}^{(i)}=0$ for $i=0,\dots,k-2$. 
\end{enumerate}

Let us assume now that $g\geqslant5$
(we will treat the case $g=4$ in the next section).
For $k=2,\dots,g-2$ there are commutative diagrams
\begin{equation}
\label{inf}
\xymatrix{
 H^{2,1}(J(C))\ar[rr]^{\hspace{.4cm}(\nabla^{g-2}_{g-1}(0))^{t}}\otimes T_{B,0} \ar[d]^{\theta^{k-1}\otimes\textit{Id}} & &
H^{1,2}(J(C))\ar[d]^{\theta^{k-1}} \\
H^{k+1,k}(J(C))\otimes T_{B,0}\ar[rr]^{\hspace{.5cm}(\nabla^{g-k-1}_{g-k}(0))^{t}} & & H^{k,k+1}(J(C))
.}
\end{equation}
Recall that $\delta\nu_{1}(0)\in\ker\big(\nabla^{g-2}_{g-1}(0)^{t}\big)^{*}$, 
\mbox{$\delta\nu_{k}(0)\in\ker\big(\nabla^{g-k-1}_{g-k}(0)^{t}\big)^{*}$}.
Let us assume now that $g\geqslant5$
(we will treat the case $g=4$ in the next section).
For $k=2,\dots,g-2$ there are commutative diagrams
\begin{equation}
\label{inf}
\xymatrix{
 H^{2,1}(J(C))\ar[rr]^{\hspace{.4cm}(\nabla^{g-2}_{g-1}(0))^{t}}\otimes T_{B,0} \ar[d]^{\theta^{k-1}\otimes\textit{Id}} & &
H^{1,2}(J(C))\ar[d]^{\theta^{k-1}} \\
H^{k+1,k}(J(C))\otimes T_{B,0}\ar[rr]^{\hspace{.5cm}(\nabla^{g-k-1}_{g-k}(0))^{t}} & & H^{k,k+1}(J(C))
.}
\end{equation}
Recall that $\delta\nu_{1}(0)\in\ker\big(\nabla^{g-2}_{g-1}(0)^{t}\big)^{*}$, 
\mbox{$\delta\nu_{k}(0)\in\ker\big(\nabla^{g-k-1}_{g-k}(0)^{t}\big)^{*}$} and the Lefschetz decomposition (\ref{lefschetz}).
Since by Proposition \ref{0} we know that $\delta\nu_{k}^{(i)}(0)=0$ for $i=0,\dots,k-2$, we actually have 
\begin{equation}
\label{boh}
\delta\nu_{k}(0)\in\big(\ker(\nabla^{g-k-1}_{g-k}(0)^{t})\cap\textrm{Im}(\theta^{k-1}\otimes\textrm{Id})\big)^{*}.
\end{equation}
\begin{lem}
$(\theta^{k-1}\otimes\textrm{Id})(\ker\nabla^{g-2}_{g-1}(0)^{t})=(\ker\nabla^{g-k-1}_{g-k}(0)^{t})\cap$ \mbox{$Im(\theta^{k-1}\otimes Id)$}.
\end{lem}
\begin{proof}
We only have to prove \[
(\ker\nabla^{g-k-1}_{g-k}(0)^{t})\cap\textrm{Im}(\theta^{k-1}\otimes \textrm{Id})\subset(\theta^{k-1}\otimes\textrm{Id})(\ker\nabla^{g-2}_{g-1}(0)^{t}),
\] 
the other inclusion being obvious.

Let us consider $\Gamma\in(\ker\nabla^{g-k-1}_{g-k}(0)^{t})\cap\textrm{Im}(\theta^{k-1}\otimes \textrm{Id})$. 
Then $\Gamma=$\mbox{$(\theta\otimes\textrm{Id})(\eta)$}
 with $\eta\in H^{2,1}(J(C))\otimes T_{B,0}$.
By commutativity of (\ref{inf}) it follows that \mbox{$\theta^{k-1}(\nabla^{g-2}_{g-1}(0))^{t}(\eta)=0$.}
\\
Hard Lefschetz theorem implies that $\theta^{g-3}:H^{3}(J(C))\rightarrow H^{2g-3}(J(C))$ is an isomorphism, which forces 
\mbox{$\theta^{k-1}:H^{1,2}(J(C))\rightarrow H^{k,k+1}(J(C))$} 
to be injective for $k=1,\dots,g-2$.
Then $(\nabla^{g-2}_{g-1}(0))^{t}(\eta)=0$.
\end{proof}
It follows from the above Lemma and (\ref{boh}) that
\[
\delta\nu_{k}(0)\in
\ker\big(\nabla^{g-k-1}_{g-k}(0)^{t}:\theta^{k-1} H^{2,1}(J(C))\otimes T_{B,0}\to H^{k,k+1}(J(C))\big)^{*}
\] 
or, equivalently, that
\[
\delta\nu_{k}(0)\circ(\theta^{k-1}\otimes\textrm{Id})\in\ker\big(\nabla^{g-2}_{g-1}(0)^{t}\big)^{*}.
\]
This vector space also contains $\delta\nu_{1}(0)$.
In view of the equality $\mathbf{ \Theta}^{k-1}\nu_{k}=C(g,k)\nu_{1}$ one might expect that also the functional above are proportional.
This is the content of the next proposition.
\begin{pro}\label{i.i.comp}
Suppose that $g\geqslant5$. Then for $k=2,\dots,g-2$ we have 
\[
\delta\nu_{k}(0)\circ(\theta^{k-1}\otimes\textrm{Id})=C(g,k)\delta\nu_{1}(0).
\]
In particular the Ceresa cycles $W_{k}-W_{k}^{-}$ are not algebraically trivial.
\end{pro}
\begin{proof}
Using the Leibniz rule for $\nabla$ and the fact that the sections of the bundle $\mathcal{H}^{2}$ annihilated by $\nabla$ are the sections of $R^{2}j(\pi)_{*}\mathbb{C}$ 
we find
$\nabla\Big(\mathbf{ \Theta}^{k-1}\tilde{\nu}_{k}\Big)=\mathbf{ \Theta}^{k-1}\nabla\tilde{\nu}_{k}$,
where $\tilde{\nu}_{k}$ is a local lifting of the normal function.
Then, by Theorem \ref{comparison}, $\mathbf{ \Theta}^{k-1}\nabla\tilde{\nu}_{k}=C(g,k)\tilde{\nu}_{1}$.
The formula follows from formula (\ref{def}).

The second statement follows from the non-triviality of $\delta\nu_{1}(0)$ \cite{cp} on some decomposable elements 
$\omega_{1}\wedge\omega_{2}\wedge\overline{\omega}_{3}\otimes\xi\in P^{2,1}(J(C))\otimes H^{1}(T_{C})$.
\end{proof}
\begin{oss}
We want to compare the above results with those of R. Hain.
Suppose $g\geqslant3$ and consider, for $l\in\mathbb{N}_{>0}$ the moduli space $\mathcal{M}_{g}(l)$ of smooth curves of genus $g$
with a fixed basis of $H_{1}(C,\mathbb{Z}/l\mathbb{Z})$.
This moduli space has a universal family $\pi:\mathcal{C}_{g}(l)\to \mathcal{M}_{g}(l)$. 
Suppose that $\mathcal{V}$ is a variation of Hodge structures over $\mathcal{M}_{g}(l)$ of weight $-1$.
One can define in this situation a bundle $\mathcal{JV}\to\mathcal{M}_{g}(l)$ whose fiber over $t$ is $JV_{t}=
\frac{\displaystyle V_{t}}{\displaystyle F^{0}V_{t}+V_{t,\mathbb{Z}}}$ and
 the variations of Hodge structures of weight $-1$ $H^{2g-2k-1}(-g+k)$ whose fiber over $t$ is
$\mathcal{H}^{2g-2k-1}(J(C_{t}),\mathbb{C})(-g+k)$.
The Lefschetz decomposition on the fibers induces a decomposition (cf. (\ref{global}))
\[
\mathcal{H}^{2g-2k-1}=\mathbf{\Theta}^{g-2k-1}\mathcal{P}^{2k+1}(-g+k)\oplus\cdots\oplus\mathbf{\Theta}^{g-k-3}\mathcal{P}^{5}(-g+k)
     \oplus\mathbf{\Theta}^{g-k-2}\mathcal{P}^{3}(-g+k) \oplus\mathbf{\Theta}^{g-k-1}\mathcal{H}^{1}(-g+k).
\]
Set $\mathcal{V}=\mathcal{H}^{2g-2k-1}$ and $\mathcal{V}(\lambda_{1})=\mathbf{ \Theta}^{g-k-1}\mathcal{H}^{1}$,
$\mathcal{V}(\lambda_{3})=\mathbf{ \Theta}^{g-k-2}\mathcal{P}^{3}$, 
$\mathcal{V}(\lambda_{2k-2i+1})=\mathbf{ \Theta}^{g-2k+i-1}\mathcal{P}^{2k-2i+1}$ for $i=0,\dots,k-2$. (cf. this notation with that of \cite{H}).
In \cite{H} it is constructed a normal function $e:\mathcal{M}_{g}(l)\to\mathcal{JV}(\lambda_{3})$ associated to the basic cycle $C-C^{-}$.
When $g\geqslant4$ one could also consider in the same way normal functions $e_{k}:\mathcal{M}_{g}(l)\to\mathcal{JV}(\lambda)$ associated to $W_{k}-W_{k}^{-}$,
$k=2,\dots,g-2$
and their components $e_{k}^{(i)}:\mathcal{M}_{g}(l)\to\mathcal{JV}(\lambda_{i})$, $i=0,\cdots,k-2$. 
We can now state Hain's results in our particular case.
\begin{teo}[\cite{H}, section 8]
\begin{enumerate}
  \item With the above notations, the group of normal functions $s:\mathcal{M}_{g}(l)\to\mathcal{JV}(\lambda_{i})$ has rank $1$ for $i=3$ and $0$ otherwise.
  \item Every normal function $s:\mathcal{M}_{g}(l)\to\mathcal{JV}(\lambda_{3})$  is a multiple of $e$.
\end{enumerate}
\end{teo}
It follows in particular that $e_{k}^{(i)}=0$ for $i\neq 3$.
In our case the normal functions $e$ and $e_{k}$ are replaced by $\nu_{1}$ and $\nu_{k}$.
Notice that they are defined for every family of curves, and that they are local objects (if $\pi:\mathcal{C}\to B$ is the Kuranishi family of a curve of genus $g$
then $B$ is an open subset of $\mathcal{M}_{g}$) while $e$ and $e_{k}$ are globally defined.
\end{oss}
\begin{dom}\label{q}
If $V$ is a smooth cubic in $\mathbb{P}^{6}$ there is an isomorphism $J(V)\cong Alb(F)$ \cite{c}, where $J(V)$ is the intermediate jacobian associated
to $H^{5}(V)$ and $F$ the Fano surface of planes contained in $V$. 
The intermediate jacobian $J(V)$ is a $21-$dimensional abelian variety with principal polarization $\theta$.
If we denote by $a:F\rightarrow Alb(F)\cong J(V)$ the Albanese map, 
we get homologically trivial cycles $a(F_{k})-a(F_{k})^{-}$ for $k\geqslant 1$ and associated normal functions $\sigma_{k}$.

Are these cycles algebraically trivial? Is there a relation between $\sigma_{1}$ and $\sigma_{k}$ similar to that between $\nu_{1}$ and $\nu_{k}$?
\\
It is shown in \cite{Jia10} that the cohomology class of $a(F)$ is $\frac{\displaystyle\theta^{19}}{19!}$.
This lead us to suspect that the analogous of Theorem \ref{abel} holds in this case.
For what concern algebraic equivalence, one can study the infinitesimal invariants.
In this case $\delta\sigma_{1}(0)\in$\mbox{$\big(\ker(\nabla(0))^{t}:H^{3,2}(J(V))\otimes H^{1}(T_{V})\rightarrow H^{2,3}\big)^{*}$},
and one can use Theorem \ref{adj} to compute it on decomposable tensors.
It is possible to show \cite{miatesi} that there exist $\xi\in H^{1}(T_{V})$ and a $3-$dimensional space $W\in H^{3,2}(V)\cong H^{1,0}(J(V))$ annihilated by 
$\xi$ such that the adjoint form is not zero.
We are not able to determine, however, if $\delta\sigma_{1}(0)$ is zero or not.
\end{dom}

\subsection{Genus 4}

Let us consider a smooth generic curve $C$ of genus $4$ and its Kuranishi family $\mathcal{C}\to B$.
In the notation of the preceding section, Theorem \ref{0} in this case would say that the section $\nu_{2}^{(1)}$ of $\mathcal{P}^{3}$ is zero.
We show in this section that this is not the case.
The proof uses an explicit computation of the adjoint form, which is the main result of this section.

Let us $C$ identify with the canonical curve in $\mathbb{P}H^{0}(C,\omega_{C})^{*}
\cong\mathbb{P}^{3}$.
It is well known that it is a complete intersection of a smooth quadric $Q$ and a cubic. 
The rulings of the quadric cut on $C$ two divisors
$D_{1}=x_{1}+x_{2}+x_{3}$ and $D_{2}=y_{1}+y_{2}+y_{3}$. 
To a point $q\in Q\setminus C$ there correspond a first order deformation $\xi_{q}\in H^{1}(T_{C})$ of rank one, which is not a Schiffer \cite{grif}. 
Let us call $V=\ker\xi_{q}=\langle\omega_{1},\omega_{2},\omega_{3}\rangle$; then 
$\mathbb{P}V=\{\textrm{hyperplanes in }\mathbb{P}^{3} \textrm{ containing }q\}$. 
Let us complete $V$ to a basis of $H^{1,0}(C)$: $H^{1,0}(C)=\langle\omega_{1},\omega_{2},\omega_{3},\omega_{4}\rangle$ and
consider $V^{2}\subset\wedge^{2}H^{1,0}(C)\cong H^{0}(\omega_{C_{2}})$ \cite{m61}.
It is easy to check that the base locus of $|V^{2}|\subset|\omega_{C_{2}}|$ is 
\[
B_{V^{2}}=Z=\{x_{1}+x_{2},x_{1}+x_{3},x_{1}+x_{3},y_{1}+y_{2},y_{1}+y_{3},y_{2}+y_{3}\}.
\]
The following lemma is due to A. Collino (personal communication).
\begin{lem}\label{collino}
We have $h^{0}(C_{2},\omega_{C_{2}}\otimes\mathcal{I}_{Z})=4$.
\end{lem}
\begin{proof}
We have a commutative diagram
\[
\xymatrix{         
C_{2} \ar[dr]^{\mathcal{G}}\ar[r] & \mathbb{P}H^{0}(\omega_{C_{2}})^{*}\cong\mathbb{P}^{5}
\\ & \mathbb{G}(1,\mathbb{P}H^{0}(\omega_{C})^{*})\cong\mathbb{G}(1,3)\ar[u]^{P}
},\]
where $\phi:C_{2}\rightarrow \mathbb{P}H^{0}(\omega_{C_{2}})^{*}$ is the bicanonical map, 
$P:\mathbb{G}(1,3)\rightarrow\mathbb{P}^{5}$ the Plucker embedding and $\mathcal{G}(D)=\overline{D}$
(we denote by $\overline{D}$ the linear span of $D$, i.e. the line $\overline{xy}$ if $D=x+y$ or the tangent line to $x$ if $D=2x$). 
If $\sigma,\eta\in H^{1,0}(C)$ then $\sigma\wedge\eta\in H^{0}(C_{2},\omega_{C_{2}}\otimes\mathcal{I}_{Z})$ iff
 $(\sigma\wedge\eta)(D)=0$ for every $D\in Z$.
Such a form defines an hyperplane $H(\sigma\wedge\eta)=\{\sigma\wedge\eta=0\}$ in $\mathbb{P}^{5}$. 
Then $\sigma\wedge\eta\in H^{0}(C_{2},\omega_{C_{2}}\otimes\mathcal{I}_{Z})$ iff $P(\overline{D})\in H(\sigma\wedge\eta)$ for every $D\in B_{V^{2}}$, 
iff $P(l_{1}),P(l_{2})\in H(\sigma\wedge\eta)$, where $l_{1},l_{2}$ are the rulings through $q$.
\\
We conclude that
$\mathbb{P}H^{0}(\omega_{C_{2}}\otimes\mathcal{I}_{Z})\cong\{\textrm{hyperplanes containing }P(l_{1}),P(l_{2})\}$.
Let us take coordinates in $\mathbb{P}^{3}$ such that $q=[0:0:0:1]$. 
If $x\in\textrm{Supp}D_{1}$, $y\in\textrm{Supp}D_{2}$, then
\begin{equation}\label{conc}
\begin{array}{l}
P(l_{1}) =  [0:0:\omega_{1}(x):0:\omega_{2}(x):\omega_{3}(x)]      \\
P(l_{2}) =  [0:0:\omega_{1}(y):0:\omega_{2}(y):\omega_{3}(y)]
\end{array}
\end{equation}
Now since $\omega_{i}\wedge\omega_{j}$ gives coordinates on $\mathbb{P}H^{0}(\omega_{C_{2}})^{*}\cong\mathbb{P}^{5}$, using (\ref{conc})
we find 
\[
\mathbb{P}H^{0}(\omega_{C_{2}}\otimes\mathcal{I}_{Z})\cong\mathbb{P}\langle\omega_{1}\wedge\omega_{2},
\omega_{1}\wedge\omega_{3},\omega_{2}\wedge\omega_{3},\Omega\rangle,
\] 
where
$\Omega\in\langle\omega_{i}\wedge\omega_{4}\rangle_{i=1,2,3}$.
\end{proof}
Recalling that $\omega_{\xi_{q},V}\in\frac{\displaystyle{H^{0}(\omega_{C_{2}}\otimes\mathcal{I}_{Z})}}{\displaystyle{V^{2}}}$, we obtain immediately
\begin{cor}\label{exp}
$\omega_{\xi_{q},V}=c\Omega$, with $c\in\mathbb{C}^{*}$.
\end{cor}
Let us fix a basis $\omega_{1},\dots,\omega_{4}\in H^{1,0}(C)$ such that the following hold:
  \\
  1) $\theta=i\sum_{j=1}^{4}\omega_{j}\wedge\overline{\omega}_{j}$,
  \\
  2) $\ker\xi_{q}=\langle\omega_{1},\omega_{2},\omega_{3}\rangle$,
  \\ 
  3) $\xi_{q}\cdot\omega_{4}=\rho\omega_{4}$, with $\rho\in\mathbb{C}$ (because $\theta\cdot\xi_{q}=0$ implies 
 $(\xi_{q}\cdot\omega_{4})\wedge\omega_{4}=0$).
 \\
 4) $\omega_{\xi_{q},\omega_{1},\omega_{2},\omega_{3}}=\omega_{1}\wedge\omega_{4}$.
\begin{pro}
The section $\nu_{2}^{(1)}$ of $\mathcal{P}^{3}$ is not torsion.
\end{pro}
\begin{proof}
Set $\phi=\omega_{1}\wedge\omega_{2}\wedge\omega_{3}$.
Then it is easy to check that $\phi\wedge\bar{\omega}_{1}\wedge\bar{\omega}_{4}\in P^{3,2}(J(C))$.
By Theorem \ref{adj} and Corollary \ref{exp}
\[
\delta\nu_{2}((\phi\wedge\overline{\omega}_{1}\wedge\overline{\omega}_{4})\otimes\xi_{q})=
2\int_{C_{2}}\omega_{\xi_{q},\omega_{1},\omega_{2},\omega_{3}}\wedge\overline{\omega}_{1}\wedge\overline{\omega}_{4}=
2\int_{C_{2}}\omega_{1}\wedge\omega_{4}\wedge\overline{\omega}_{1}\wedge\overline{\omega}_{4}\neq 0.
\]
\end{proof}

\section{The zero locus of $\delta\nu_{2}$}\label{i.i.} 
In this section we determine the zero-locus of $\delta\nu_{2}(0)$. 
We will deduce from that the Torelli theorem for curves of genus $g\geqslant4$.
Although not explicetely stated in \cite{cp}, the Torelli theorem for genus $g\geqslant3$ can also be deduced from the computation of $\delta\nu_{1}$ 
on decomposable elements $\tau_{1}\wedge\tau_{2}\wedge\eta\otimes v\in P^{2,1}(J(C))\otimes T_{B,0}$.  
One can then appeal to Theorem  \ref{i.i.comp} to reconstruct the curve from $\delta\nu_{2}(0)$.
The interesting thing is that we will use Theorem \ref{adj} to compute $\delta\nu_{2}(0)$ on some decomposable tensors 
$\omega_{1}\wedge\omega_{2}\wedge\omega_{3}\wedge\sigma_{1}\wedge\sigma_{2}\otimes v\in H^{3,2}(J(C))\otimes T_{B,0}$ 
which are not, in general, of the form $\theta\wedge\tau_{1}\wedge\tau_{2}\wedge\eta\otimes v$.
We then get extra geometric informations from $\delta\nu_{2}(0)$ which cannot be detected from the computations of \cite{cp}. 
\subsection{Holomorphic forms on $C_{2}$}\label{holo}
Suppose that $C$ is a non-hyperelliptic curve.
We begin by recalling well-known facts about holomorphic one forms on $C_{2}$. 
Recall that $C_{2}=\frac{\displaystyle{C\times C}}{\displaystyle{S_{2}}}$, where $S_{2}$ is the symmetric group of order two, generated by the transposition $\sigma$. 
We then have $H^{1,0}(C_{2})\cong H^{1,0}(C\times C)^{\textit{inv}}$, where $H^{1,0}(C\times C)^{\textit{inv}}$ is the space of invariant forms. 
By the Kunneth decomposition we have $H^{1,0}(C\times C)\cong H^{1,0}(C)\otimes H^{1,0}(C)$, and the action of $\sigma$ interchanges the two factors.
Then 
\[
H^{1,0}(C_{2})\cong \{(\omega_{1},\omega_{2})\in H^{1,0}(C)\times H^{1,0}(C) \ | \ \omega_{1}=\omega_{2}\}\cong H^{1,0}(C).
\]
For the remaining part of the paper we will adopt the following notation: if $\Omega\in H^{1,0}(C_{2})$, we will denote by $\omega$ the corresponding form in
$H^{1,0}(C)$ under the isomorphism above, and viceversa.

We also remark that there is an isomorphism $H^{2,0}(C_{2})\cong\bigwedge^{2}H^{0}(C,\omega_{C})$ \cite{m61}.
Now we turn to the zero locus of holomorphic forms.
\begin{lem}\label{zero}
If $\Omega\in H^{1,0}(C_{2})$ then $Z(\Omega)=\{D\in C_{2}\textrm{ }|\textit{ }\omega\in H^{0}(\omega_{C}(-D))\}$.
\end{lem}
\begin{proof}
Let us consider the case $D=p+q$ with $p$ and $q$ distinct. 
Then in a neighborhood of $D$ we can take as local coordinates the same as those in $C\times C$.
If $\eta\in H^{1,0}(C\times C)$ corresponds to $(\eta_{1},\eta_{2})\in H^{1,0}(C)\otimes H^{1,0}(C)$ under the Kunneth isomorphism,
 then $\eta(p,q)=0$ iff $\eta_{1}(p)=\eta_{2}(q)=0$.
Now since $H^{1,0}(C_{2})$ is the space of invariant forms under the transposition, for $\Omega\in H^{1,0}(C_{2})$ we have:
\[
\Omega(D)=0 \Leftrightarrow\omega(p)=\omega(q)=0.
\]
In a similar way it is possible check that the result holds for $D=2p$.  
\end{proof}
We now consider zeroes of holomorphic two-forms. 
If $\sigma\wedge\eta\in\bigwedge^{2}H^{0}(\omega_{C})\cong H^{2,0}(C_{2})$ then, by definition of wedge product, $\sigma\wedge\eta(D)=0$ 
iff there exists $\alpha,\beta\in\mathbb{C}$, not both zero, such that \mbox{$\alpha\sigma(D)+\beta\eta(D)=0$.}
Then by the above lemma, for $D=p+q$
\begin{equation}\label{glizeri}
(\sigma\wedge\eta)(D)=0\Leftrightarrow\exists\alpha,\beta\in\mathbb{C}\textrm{ not both zero such that }
\left\{
\begin{array}{l}
    \alpha\sigma(p)+\beta\eta(p)=0  \\
     \alpha\sigma(q)+\beta\eta(q)=0
\end{array}
\right.
.\end{equation}
Let us consider three indipendent holomorphic 1-forms $\Omega_{1},\Omega_{2},\Omega_{3}\in H^{1,0}(C_{2})$, and 
\mbox{$W=\langle\Omega_{1},\Omega_{2},\Omega_{3}\rangle$}. 
We also denote by $V$ the corresponding subspace of $H^{1,0}(C)$: $V=\langle\omega_{1},\omega_{2},\omega_{3}\rangle\subset H^{1,0}(C)$. 
We are interested in the base locus $B_{W^{2}}$ of the linear system $|W^{2}|$, where
$W^{2}=\bigwedge^{2}W
\subset H^{2,0}(C_{2})$. 
\begin{pro}\label{baselocus}
\begin{enumerate}
  \item  If $V$ has a base point $p\in C$, then the curve $C_{p}=\{p+q\in C_{2} \ | \ q\in C\}$ is contained in the base locus $B_{W^{2}}$.
  \item  Suppose that $V$ is base-point-free and that the induced map
  $\phi_{V}:C\to\mathbb{P}V^{*}\cong\mathbb{P}^{2}$ is birational onto the image. 
  Then $B_{W^{2}}$ is finite.
\end{enumerate}
 \end{pro}
\begin{proof}
\begin{enumerate}
  \item  Let $p$ be a base point of $V$. 
For any $q\in C$ we consider the evaluation morphism $ev_{q}:V\otimes \mathcal{O}_{C}\rightarrow \mathbb{C}_{q}$.
Then $\ker(ev_{q})$ has rank two, and we may suppose, up to a change of basis, that it is generated by $\omega_{1},\omega_{2}$. 
Let us consider $D=p+q\in C_{p}$. We want to show that $(\Omega_{i}\wedge \Omega_{j})(D)=0$ for $i,j=1,2,3$.
\\
Recalling (\ref{glizeri}) we see that $(\Omega_{i}\wedge \Omega_{j})(D)=0$ is certainly
satisfied since $\omega_{i}(p)=0$, $i=1,2,3$ and $\omega_{1}(q)=\omega_{2}(q)=0$.
  \item Let us consider the map
\[\begin{array}{llll }
      \phi_{V}: & C & \rightarrow & \mathbb{P}V^{*}\cong\mathbb{P}^{2}     \\
      &    p & \mapsto & [\omega_{1}(p):\omega_{2}(p):\omega_{3}(p)].
\end{array}
 \]
and define $C^{\prime}=\phi_{V}(C)$. 
Notice that we have a 1-1 correspondence
\[ \{D=p+q\in B_{W^{2}}, p\neq q\}\leftrightarrow \{\textrm{nodes of }C^{\prime}\}. \]
In fact
\[ D=p+q\in B_{W^{2}} \textrm{ if and only if }  \left\{ \begin{array}{ll}
\omega_{i}(p)=\lambda_{i}\omega_{1}(p) & \textmd{ $i=2,3$} \\
\omega_{i}(q)=\lambda_{i}\omega_{1}(q) & \textmd{ $i=2,3$} 
\end{array}
\right
.\]
for some $\lambda_{2},\lambda_{3}\in\mathbb{C}$.
But this means that 
\[
\begin{array}{ll}
\phi_{V}(p)=[\omega_{1}(p):\lambda_{2}\omega_{1}(p):\lambda_{3}\omega_{1}(p)]=[1:\lambda_{2}:\lambda_{3}]    &    \\
\phi_{V}(q)=[\omega_{1}(q):\lambda_{2}\omega_{1}(q):\lambda_{3}\omega_{1}(q)]=[1:\lambda_{2}:\lambda_{3}]      &   
\end{array}
,\]
i.e. $\phi_{V}(p)$ is a node (notice that we are using that $V$ is base-point-free).
\\
We also have a 1-1 correspondence \[
\{D=2p\in B_{W^{2}}\}\leftrightarrow \{\textrm{cusps of }C^{\prime}\}.
\]
To see this, recall that $(\omega_{i}\wedge\omega_{j})(2p)=0$ if and only if $\lambda_{i}\omega_{i}+\mu_{j}\omega_{j}\in H^{0}(\omega_{C}(-2p))$ for every $i,j$.
Let $D=2p$ be in $B_{W^{2}}$. It is not restrictive to assume that $\omega_{1},\omega_{2}\in H^{0}(\omega_{C}(-2p))$ and $\omega_{3}(p)\neq 0$.
In local coordinate near $p$ we have $\omega_{i}=f_{i}(z)dz$, $i=1,2$, with $f_{i}(0)=f_{i}^{\prime}(0)=0$. 
We also have $\phi_{V}(z)=(f_{1}(z),f_{2}(z))$ so that $d_{p}\phi_{V}=(f_{1}^{\prime}(0),f_{2}^{\prime}(0))=(0,0)$, i.e. $p$ is a cusp. 
\end{enumerate}
\end{proof}
\begin{oss}
If $C$ is general and $V$ has no base points then $\phi_{V}$ is birational to the image.
To see this suppose that it is not. 
Then by \cite[cor. 8.32 p.834]{acg2} the morphism factorizes as 
$\phi_{V}=\nu\circ g$, where $\nu:E\rightarrow C^{\prime}$ is the normalization, $E\cong\mathbb{P}^{1}$ and $g:C\rightarrow E$ is holomorphic. 
We have $\deg\phi_{V}=\deg g$ and since $C$ is not hyperelliptic $\deg g\geqslant 3$. If we denote $L=g^{*}\mathcal{O}_{\mathbb{P}^{1}}(1)$, we have
\[
 \omega_{C}\cong\phi_{V}^{*}\mathcal{O}_{C^{\prime}}(1)=g^{*}\nu^{*}\mathcal{O}_{C^{\prime}}(1)=g^{*}\mathcal{O}_{\mathbb{P}^{1}}(k)=L^{\otimes k},
  \]
for some $k\in \mathbb{Z}_{>0}$.
\\
The Petri map associated to $L$ is $\mu:H^{0}(L)\otimes H^{0}(L^{\otimes (k-1)})\rightarrow H^{0}(L^{\otimes k})$. 
If $s,t\in H^{0}(L)$, then $t\otimes s^{\otimes(k-1)}-s\otimes (s^{\otimes(k-2)}\otimes t)\in\ker\mu$, contradicting the injectivity of $\mu$ 
\cite[prop. 3.20 pag.795]{acg2}.
\end{oss}
\begin{oss}\label{-2p}
Arguing as in the proof of 1) above, one can show that if $V\subset H^{0}(\omega_{C}(-2p))$, then $2C_{p}\subset B_{W^{2}}$.
\end{oss}

\subsection{Infinitesimal variations of Hodge structures for $C_{2}$}
In the next section we will study the adjoint form on $C_{2}$. 
In order to construct it we need to study the infinitesimal variations of Hodge structures of $C_{2}$ and the varieties attached to them.
That is the content of this section.

First notice that under the isomorphism $H^{1,0}(C_{2})\cong H^{1,0}(C)$
and $H^{1}(T_{C_{2}})\cong H^{1}(T_{C})$ \cite[Thm. 10.1 pag. 242]{acg2}, the infinitesimal variations of Hodge structures of $C_{2}$ 
correspond to the ones of $C$.

Let us consider 
\[ \Gamma=\{(\xi,V)\textrm{ }|\textrm{ }V\subset\ker\xi\}\subset \mathbb{P}H^{1}(T_{C})\times G(3,H^{1,0}(C)),
\]
with  the projections $\pi_{1},\pi_{2}$ and $\Sigma=\pi_{2}(\Gamma)$. Note that 
\[
\pi_{2}^{-1}(V)=\{\xi\in H^{1}(T_{C})\textrm{ }|\textrm{ }\xi\cdot V=0\}.
\]
There is an isomorphism 
\begin{equation}\label{iso}
\begin{array}{ cll}
      \mathbb{G}(g-4,\mathbb{P}H^{0}(\omega_{C})^{*}) & \rightarrow & G(3,H^{1,0}(C)) \\
        L=\{\omega_{1}=\omega_{2}=\omega_{3}=0\} & \mapsto & V(L):=\langle\omega_{1},\omega_{2},\omega_{3}\rangle
\end{array}
\end{equation}

In the sequel we will always identify $\mathbb{P}H^{0}(C,\omega_{C})^{*}$ with the canonical space $\mathbb{P}^{g-1}$ and $C$
with its canonical image $\phi_{\omega_{C}}(C)$.
\\
Let us define $\mathcal{D}=\{W\in G(3,H^{1,0}(C)) \textrm{ with base points}\}$. 
Notice that $\dim\mathcal{D}=\dim G(3, H^{0}(\omega_{C}(-p)))+1=3(g-4)+1=3g-11$.
Geometrically, a point $W\in\mathcal{D}$ correspond under the above isomorphism to a $(g-4)$-dimensional linear subspace of $\mathbb{P}^{g-1}$ 
intersecting $C$.
\begin{lem}
\label{ivhs}
\begin{enumerate}
  \item $\dim\Gamma=\dim\Sigma=3g-10$.
  \item $\mathcal{D}\subset\Sigma$.
\end{enumerate}
 \end{lem} 
\begin{proof}
\begin{enumerate}
  \item Notice that there exists $V\in G(3,H^{1,0}(C))$ such that $\dim\Ann V=1$ 
(see for example \cite[Lemma 2.4]{r}). 
This means that the generic fibre of $\pi_{2}$ has dimension zero, hence $\dim\Sigma=\dim\Gamma$.
To see that the dimension is $3g-10$ it suffices to recall that $\dim G(3,H^{1,0}(C))=3g-9$ and $\Sigma$ contains a divisor \cite[p.570]{pz}.
 \item Let us suppose that $W\in\mathcal{D}$ and let $p$ be a base point.
Recall that the Schiffer variation $\theta_{p}$ has kernel $H^{0}(C,\omega_{C}(-p))$ \cite{grif}. 
Then the Schiffer variation $\theta_{p}$ annihilates $W$.
\end{enumerate}
\end{proof}

\subsection{Computation of the adjoint class}
\label{tecnico}
We want to determine where does the adjoint class vanish.
From now on we will assume that $C$ is a generic curve.
For the convenience of the reader we first recall some notation.
Let us consider $\xi\in H^{1}(C,T_{C})\cong H^{1}(C_{2},T_{C_{2}})$;
then $\xi$ defines two first order deformations $\mathcal{C}\to\textrm{Spec}\frac{\displaystyle\mathbb{C}[t]}{\displaystyle(t^{2})}$
and $\mathcal{C}_{2}\to\textrm{Spec}\frac{\displaystyle\mathbb{C}[t]}{\displaystyle(t^{2})}$ of $C$ and $C_{2}$ respectively.
The associated extension classes are
\begin{equation}
0\to\mathcal{O}_{C}\to\Omega^{1}_{\mathcal{C}|C}\to\omega_{C}\to0
\end{equation}
\begin{equation}
\label{c2}
0\to\mathcal{O}_{C_{2}}\to\Omega^{1}_{\mathcal{C}_{2}|C_{2}}\to\Omega^{1}_{C_{2}}\to0.
\end{equation}
We are now ready to state and prove the following technical lemma.
\begin{lem}
\label{tech}
Let $\omega_{1},\omega_{2},\omega_{3}\in H^{1,0}(C)$ be independent and $\Omega_{1},\Omega_{2},\Omega_{3}$ be the corresponding forms in $H^{1,0}(C_{2})$.
Suppose that there exists $\xi\in H^{1}(C,T_{C})\cong H^{1}(C_{2},T_{C_{2}})$ such that $\xi\cdot\omega_{i}=0$ for $i=1,2,3$ and the adjunction
classes $[\omega_{\xi,\omega_{k},\omega_{l}}]$ on $C$ vanishes for every $k,l$.
\\
Then, for every $q\in C$, there exists some representative of $[\omega_{\xi,\Omega_{1},\Omega_{2},\Omega_{3}}]$ that vanishes on $C_{q}=\{x+q \ | \ x\in C\}$.
\end{lem}
\begin{proof}
By hypothesis the forms $\omega_{1},\omega_{2},\omega_{3}$ lift to $H^{0}(C,\Omega^{1}_{\mathcal{C}|C})$.
\begin{claim}
 It is possible to choose liftings $\widetilde{\omega}_{1},\widetilde{\omega}_{2},\widetilde{\omega}_{3}\in H^{0}(C,\Omega^{1}_{\mathcal{C}|C})$
such that \mbox{$\widetilde{\omega}_{k}\wedge\widetilde{\omega}_{l}=0$} in $H^{0}(C,\bigwedge^{2}\Omega^{1}_{\mathcal{C}|C})$ for every $k,l$.
\end{claim}
\begin{proof}
Recall that $[\omega_{\xi,\omega_{k},\omega_{l}}]\in\frac{\displaystyle H^{1,0}(C)}{\displaystyle\langle\omega_{k},\omega_{l}\rangle}$ and that,
given liftings $\widetilde{\omega}_{1},\widetilde{\omega}_{2},\widetilde{\omega}_{3}$, a representative of $[\omega_{\xi,\omega_{k},\omega_{l}}]$ is 
$L_{\xi}(\widetilde{\omega}_{k}\wedge\widetilde{\omega}_{l})$ (see Section \ref{adjunction}).
Then by hypothesis the adjoint classes on $C$ have representatives
\begin{align}
 L_{\xi}(\widetilde{\omega}_{1}\wedge\widetilde{\omega}_{2})=\alpha_{1}\omega_{1}+\alpha_{2}\omega_{2}     \\
L_{\xi}(\widetilde{\omega}_{1}\wedge\widetilde{\omega}_{3})=\beta_{1}\omega_{1}+\beta_{3}\omega_{3} \\
\label{terza}
L_{\xi}(\widetilde{\omega}_{2}\wedge\widetilde{\omega}_{3})=\gamma_{2}\omega_{2}+\gamma_{3}\omega_{3}
\end{align}
with the $\alpha$'s, $\beta$'s and $\gamma$'s complex numbers.
One can choose $\widetilde{\omega}_{1},\widetilde{\omega}_{2},\widetilde{\omega}_{3}$ in such a way that $L_{\xi}(\widetilde{\omega}_{1}\wedge\widetilde{\omega}_{2})=0$ and $\beta_{1}=0$.
This means that $\widetilde{\omega}_{1}\wedge\widetilde{\omega}_{2}=0$ in $H^{0}(C,\bigwedge^{2}\Omega^{1}_{\mathcal{C}|C})$ or, equivalently, that
$\widetilde{\omega}_{2}=f\widetilde{\omega}_{1}$ for some meromorphic function $f$ on $C$.
It is easy to see that $f=\frac{\omega_{2}}{\omega_{1}}$.
Then 
\[
L_{\xi}(\widetilde{\omega}_{2}\wedge\widetilde{\omega}_{3})=fL_{\xi}(\widetilde{\omega}_{1}\wedge\widetilde{\omega}_{3})=f(\beta_{3}\omega_{3}).
\]
Recalling (\ref{terza}) we find $\gamma_{2}\omega_{2}+\gamma_{3}\omega_{3}=\beta_{3}f\omega_{3}$.
If we multiply by $\omega_{1}$ we get the following equation in $H^{0}(C,\omega^{\otimes2}_{C})$
\[
\gamma_{2}\omega_{1}\omega_{2}+\gamma_{3}\omega_{1}\omega_{3}=\beta_{3}\omega_{2}\omega_{3}
.\]
This is the equation of a rank $3$ quadric containing the canonical curve.
Since $C$ is generic there are no such quadrics, then $\gamma_{2}=\gamma_{3}=\beta_{3}=0$ and we have proved our claim.
In particular $\widetilde{\omega}_{3}=g\widetilde{\omega}_{1}$, where $g=\frac{\omega_{3}}{\omega_{1}}$. 
\end{proof}
From now on we fix the liftings $\widetilde{\omega}_{1},\widetilde{\omega}_{2},\widetilde{\omega}_{3}$ found above.

Now let us consider the extension (\ref{c2}).
Restriction to $C_{q}$ yeld an element $\xi_{C}\in H^{1}(C_{q},T_{C_{2}|C_{q}})$.
Compatibility of the cup product with the isomorphisms $H^{1}(C,T_{C})\cong H^{1}(C_{2},T_{C_{2}})$, $H^{1,0}(C)\cong H^{1,0}(C_{2})$ enables to find 
liftings $\widetilde{\Omega}_{i}\in H^{0}(C_{2},\Omega^{1}_{\mathcal{C}_{2}|C_{2}})$ mapping to $\widetilde{\omega}_{i}$ under the composition 
$H^{0}(C_{2},\Omega^{1}_{\mathcal{C}_{2}|C_{2}})\to H^{0}(C_{q},\Omega^{1}_{\mathcal{C}_{2}|C_{q}})\to H^{0}(C_{q},\Omega^{1}_{\mathcal{C}|C_{q}})$.

Let us take local coordinates $(z,w)$ on $C_{2}$ such that $\{w=0\}$ is a local equation for $C_{q}$ in $C_{2}$.
Then, locally, 
\[
\widetilde{\omega}_{1}=\alpha(z)dz+h(z)dt, \textrm{  }\  \ \widetilde{\omega}_{2}=f\big(\alpha(z)dz+h(z)dt\big),\  \ \widetilde{\omega}_{3}=g\big(\alpha(z)dz+h(z)dt\big) 
\]
and 
\begin{align}
\label{}
   \widetilde{\Omega}_{1}=\alpha(z,w)dz+\beta(z,w)dw+h(z,w)dt, & \ \ \widetilde{\Omega}_{2}=f\big(\alpha(z,w)+\beta(z,w)dw+h(z,w)dt\big), \\ 
\widetilde{\Omega}_{3}=g\big(\alpha(z,w)dz+\beta(z,w)dw+h(z,w)dt\big), & &
 \end{align}
 where $\alpha(z,0)=\alpha(z)$ and $h(z,0)=h(z)$.
From the expression of the adjoint class in local coordinates (see Section \ref{adjunction}) we see that 
$L_{\xi}(\widetilde{\Omega}_{1}\wedge\widetilde{\Omega}_{2}\wedge\widetilde{\Omega}_{3})=0$.
\end{proof}
We can now prove the main result of this section.
\begin{pro}
\label{vanish}
Let us suppose that $W\in G(3,H^{0}(\omega_{C}(-p)))$ and let $\xi=\theta_{p}$ be the Schiffer variation of $p$.
Then $\omega_{\xi,W}=0$.
\end{pro}
\begin{proof}
Let us fix a basis $\omega_{1},\omega_{2},\omega_{3}$ of $W$ and take $\alpha_{1},\dots,\alpha_{g-4}$ that complete 
the $\omega$'s to a basis of $H^{0}(C,\omega_{C}(-p))$.
Then choose $\sigma\in H^{0}(C,\omega_{C})$ not vanishing in $p$.
If $\Omega_{1},\Omega_{2},\Omega_{3}\in H^{1,0}(C_{2})$ correspond to the $\omega$'s, choosing liftings $\widetilde{\Omega}_{1},\widetilde{\Omega}_{2},\widetilde{\Omega}_{3}$
we get a representative of $[\omega_{\xi,\Omega_{1},\Omega_{2},\Omega_{3}}]$: 
\[
L_{\xi}(\widetilde{\Omega}_{1}\wedge\widetilde{\Omega}_{2}\wedge\widetilde{\Omega}_{3})=\Upsilon+\Phi+\Psi+\Gamma
,\]
where $\Upsilon=\sum_{i,j=1}^{3}a_{ij}\omega_{i}\wedge\omega_{j}$, $\Phi=\sum_{k\leqslant3,l\leqslant g-4}b_{kl}\omega_{k}\wedge\alpha_{l}$,
$\Psi=\sum_{s,t=1}^{g-4}c_{st}\alpha_{s}\wedge\alpha_{t}$ and $\Gamma=\sigma\wedge(\sum_{r=1}^{3}d_{r}\omega_{r}+\sum_{m=1}^{g-4}e_{m}\alpha_{m})$.
We remark that $\Phi,\Psi,\Gamma$ are independent on the chosen liftings.
\medskip
\\
\textit{Step 1. }$\Gamma=0$
\\
We claim that we are in the hypothesis of Lemma \ref{tech}.
In fact $W$ has base point $p$ and the Schiffer variation $\xi$ is in the kernel of $H^{1}(C,T_{C})\to H^{1}(C,T_{C}(p))$ (cf. \cite[p. 273]{grif}).
Then by Remark \ref{curves} the adjoint classes $[\omega_{\xi,\omega_{i},\omega_{j}}]$ vanish for $i,j=1,2,3$.
\\
It follows from Lemma \ref{tech} that some representative of $[\omega_{\xi,\Omega_{1},\Omega_{2},\Omega_{3}}]$ vanishes on $C_{p}$.
Up to a change of the liftings, that will vary the coefficients of $\Upsilon$, we can suppose that
 $L_{\xi}(\widetilde{\Omega}_{1}\wedge\widetilde{\Omega}_{2}\wedge\widetilde{\Omega}_{3})$ vanishes on $C_{p}$.
Since this is also the case for $\Upsilon,\Phi,\Psi$,
we see that $\Gamma$ also vanishes on $C_{p}$. 

Let us consider a point $x+p\in C_{p}$ and recall from (\ref{glizeri}) that 
\[
\Gamma(x+p)=0 \Leftrightarrow
\lambda\sigma+\mu\big(\sum_{r=1}^{3}d_{r}\omega_{r}+\sum_{m=1}^{g-4}e_{m}\alpha_{m}\big)\in H^{0}(C,\omega_{C}(-x-p)), \textrm{ \  
$\lambda,\mu\in\mathbb{C}$ not both zero.}
\]
Then, since $\sum_{r=1}^{3}d_{r}\omega_{r}+\sum_{m=1}^{g-4}e_{m}\alpha_{m}$ vanishes in $p$ and $\sigma$ does not, we get that $\lambda=0$ in the above expression.
This implies that $\sum_{r}d_{r}\omega_{r}(x)+\sum_{m}e_{m}\alpha_{m}(x)=0$.
Since this hold for every $x\in C$ we get that all $d_{r}$'s and $e_{m}$'s are zero.
\medskip
\\
\textit{Step 2. }$\Phi=0$.
\\
Choose $p_{1},p_{2},p_{3},q_{1},\dots,q_{g-4}\in C$ generic points such that
\begin{enumerate}
  \item $\alpha_{i}(p_{j})=0$ for every $i,j$
  \item $\omega_{i}(p_{j})=0$ iff $i\neq j$
  \item $\alpha_{i}(q_{j})=0$ iff $i\neq j$ and $\omega_{3}(q_{j})\neq 0$ for every $j$.
\end{enumerate}
Then $\Psi$, $\omega_{i}\wedge\alpha_{j}$ vanish on $C_{p_{1}}$ for $i=2,3$, $j=1,\dots,g-4$.
We get 
\[
L_{\xi}(\widetilde{\Omega}_{1}\wedge\widetilde{\Omega}_{2}\wedge\widetilde{\Omega}_{3})_{|C_{p_{1}}}=
\Big(\omega_{1}\wedge(a_{12}\omega_{2}+a_{23}\omega_{3}+\sum b_{1j}\alpha_{j})\Big)_{|C_{p_{1}}}.
\]
As before we can suppose that  $L_{\xi}(\widetilde{\Omega}_{1}\wedge\widetilde{\Omega}_{2}\wedge\widetilde{\Omega}_{3})_{|C_{p_{1}}}=0$.
The form $a_{12}\omega_{2}+a_{23}\omega_{3}+\sum_{j=1}^{g-4} b_{1j}\alpha_{j}$ vanishes in $p_{1}$ while $\omega_{1}$ does not. 
Exactly as in Step 1 we conclude that $b_{1j}=0$ for all $j$.
In the same way, using $p_{2}$ and $p_{3}$, one gets that $b_{2j}=b_{3j}=0$.
\medskip
\\
\textit{Step 3. }$\Psi=0$.
\\
Let $p_{1},p_{2},p_{3},q_{1},\dots,q_{g-4}\in C$ as before.
We take another basis $\{\eta_{1},\eta_{2},\eta_{3}\}$ of $W$ defined as 
$\eta_{i}=\omega_{3}(q_{1})\omega_{i}-\omega_{i}(q_{1})\omega_{3}$ for $i=1,2$ and $\eta_{3}=\omega_{3}$.
Notice that $\eta_{1}(q_{1})=\eta_{2}(q_{1})=0$.
Reasoning as in Step 2 we see that
\[
\Big(\eta_{3}\wedge(a_{13}\eta_{1}+a_{23}\eta_{2})+\alpha_{1}\wedge(\sum_{j=1}^{g-4} c_{1j}\alpha_{j})\Big)_{|C_{q_{1}}}=0.
\] 
Since $\eta_{3}\wedge\eta_{1}$, $\alpha_{1}\wedge(\sum_{j=1}^{g-4} c_{1j}\alpha_{j})$ vanish on $q_{1}+p_{2}$ and $\eta_{3}\wedge\eta_{2}$ 
does not we must have $a_{23}=0$.
A similar argument involving $q_{1}+p_{1}$ shows that $a_{13}=0$.
Then \mbox{$\Big(\alpha_{1}\wedge(\sum_{j=1}^{g-4} c_{1j}\alpha_{j})\Big)_{|C_{q_{1}}}=0$} which implies $c_{1j}=0$ for $j=1,\dots,g-4$.
\\
Using the other points $q_{2},\dots,q_{g-4}$ one can show in the same way that $c_{ij}=0$.
\end{proof}
We now explain how the infinitesimal invariant $\delta\nu_{2}(0)$ gives a section of a suitable vector bundle on $\Gamma$.
Let us call $S$ the tautological bundle over $G(g-3,H^{1,0}(C))$.
The infinitesimal invariant, according to Theorem \ref{adj}, can be interpreted as the following morphism of vector bundles, which we call again 
$\delta\nu_{2}(0)$ by abuse of notation:
\[
\delta\nu_{2}(0): \big(\pi_{1}^{*}\mathcal{O}_{\mathbb{P}H^{1}(T_{C_{2}})}(-1)\otimes\pi_{2}^{*}\bigwedge^{3}S\big)_{|\Gamma}\otimes\big(\bigwedge^{2}H^{0}(C,\omega_{C})\otimes\mathcal{O}_{\Gamma}\big)\rightarrow
\mathcal{O}_{\Gamma},\]
which on the fiber over a point $(\xi,W)\in\Gamma$ is defined as follows. 
 Let us fix a basis $\{\omega_{1},\omega_{2},\omega_{3}\}$ of $W$. 
 Then
\[
\begin{array}{lrrc}
\delta\nu_{2}(0): & \langle\xi\rangle\otimes\bigwedge^{3}W\otimes\bigwedge^{2}H^{0}(C,\omega_{C}) & \rightarrow &
\mathbb{C} \\
& \xi\otimes\omega_{1}\wedge\omega_{2}\wedge\omega_{3}\otimes\sigma_{1}\wedge\sigma_{2} & \mapsto & 
2\int_{C_{2}}\omega_{\xi,\Omega_{1},\Omega_{2},\Omega_{3}}\wedge\overline{\sigma}_{1}\wedge\overline{\sigma}_{2} 
\end{array}
.\]
If we call $E=\big(\pi_{1}^{*}\mathcal{O}_{\mathbb{P}H^{1}(T_{C_{2}})}(-1)\otimes\wedge^{3}S\big)_{|\Gamma}\otimes
\big(\bigwedge^{2}H^{0}(C,\omega_{C})\otimes\mathcal{O}_{\Gamma}\big)$, the dual
 of the above morphism gives a section $\alpha\in H^{0}(\Gamma,E^{*})$.
We call $Z$ its zero locus.
Recall also the variety $\mathcal{D}=\{W\in G(3,H^{1,0}(C)) \textrm{ with base points}\}$.
We can summarize the previous results in the following theorem 
\begin{teo}\label{torelli} 
 $\pi_{2}(Z)=\mathcal{D}$. 
\end{teo} 
\begin{proof}
By Proposition \ref{vanish} the infinitesimal invariant vanishes on $\mathcal{D}$. 
If $W\notin\mathcal{D}$ then it has no base points.
Proposition \ref{baselocus} together with Proposition \ref{adj} then implies that for every $\xi\in\pi_{2}^{-1}(W)$ we have $\omega_{\xi,W}\neq0$.
\end{proof}
\begin{cor}
If $g\geqslant4$ the infinitesimal invariant $\delta\nu_{2}(0)$ reconstructs the curve $C$.
\end{cor}
\begin{proof}
We have to show that one can recover the curve $C$ from $\mathcal{D}$.

Let us identify the canonical space with $\mathbb{P}^{g-1}$ and consider $C\subset\mathbb{P}^{g-1}$.
Then, under the identification $G(3, H^{1,0}(C))\cong\mathbb{G}(g-4,g-1)$ we have 
\[
\mathcal{D}=\{L\in\mathbb{G}(g-4,g-1) \ | \ L\cap C\neq\emptyset\}.
\]
Let us assume first that $g=4$.
Then, by the very definition of $\mathcal{D}$, we have that $\mathcal{D}$ is the canonical curve $C$.

Let us now consider $g\geqslant5$.
Suppose there exists a canonical curve $C^{\prime}\subset\mathbb{P}^{g-1}$ such that $\mathcal{D}^{\prime}=\mathcal{D}$, where
$\mathcal{D}^{\prime}=\{L\in\mathbb{G}(g-4,g-1) \ | \ L\cap C^{\prime}\neq\emptyset\}.$
We want to show that actually $C^{\prime}=C$.
\\
Let us suppose by contradiction that there exists a point $p\in C\setminus C^{\prime}$ and consider the join of $p$ and $C^{\prime}$, i.e.
$S=J(p,C^{\prime})=\overline{\bigcup_{q\in C^{\prime}}\overline{pq}}$; we have $\dim(S)=2$.
Let $L$ be a linear space of dimension $g-4$ passing through $p$; in particular $L\in\mathcal{D}$.
Since $\mathcal{D}=\mathcal{D}^{\prime}$, there is a point $q\in L\cap C^{\prime}$ and then $\overline{pq}\subset L\cap S$.
\\
We conclude that $S$ has the following property: for every linear space $L$ of dimension $g-4$ passing through $p$ the intersection $L\cap S$ contains a line.
But this is in contradiction with the fact that $\dim(S)=2$. 
In fact one can take two hyperplanes $H_{1},H_{2}$ passing through $p$ such that $S$ is not contained in $H_{1}$ and $S\cap H_{1}$ is not contained in 
$H_{2}$.
Then $S\cap H_{1}\cap H_{2}$ is finite and for every linear space $L\subset H_{1}\cap H_{2}$ of dimension $g-4$ passing through $p$, $L\cap S$ is finite.
\end{proof}

\section{Higher K-theory on hyperelliptic Jacobians}
\subsection{Higher Chow groups and regulators}\label{hch}
We briefly recall some basic definitions. 
We refer to \cite{gl} and \cite{v} for details. 

Let $X$ be a smooth projective variety of dimension $n$. 
An element of the group $CH^{p}(X,1)$ is represented by a formal sum 
$\sum_{i}(Z_{i},f_{i})$, where $Z_{i}$ is an irreducible subvariety of $X$ of codimension $p-1$ and $f_{i}$ are rational functions on $Z_{i}$
such that $\sum_{i}\textrm{div}(f_{i})=0$ as a codimension $p$ cycle on $X$ (relations in this group are given by the tame symbols, see \cite{gl}). 
We call \textit{higher cycles} the elements of these groups.
 There is an obvious map
\[
\alpha:CH^{p-1}(X)\otimes\mathbb{C}^{*}\rightarrow CH^{p}(X,1),
\]
which maps $(\sum_{i}Z_{i})\otimes c$ to $\sum_{i}(Z_{i},c)$. 
The elements in the image of this map are called \textit{decomposable}.
The group of \textit{indecomposable Higher Chow cycles} is by defintion the quotient $CH^{p}_{ind}(X,1)=CH^{p}(X,1)/\textrm{Im}\alpha$.
We will call \textit{indecomposable} every non trivial element of this group. 
As algebraically non trivial cycles can be detected by the Abel Jacobi map, indecomposable elements can be detected by means of another map,
 called \textit{regulator}, which we introduce below. 

Let $Z=\sum_{i}(Z_{i},f_{i})$ be an element of $CH^{p}(X,1)$. 
We can view each $f_{i}$ as a rational function $f_{i}:Z_{i}\dashrightarrow\mathbb{P}^{1}$ and 
consider $\gamma_{i}=f_{i}^{-1}([0,\infty])$, where $[0,\infty]$ denotes any path 
joining $0$ and $\infty$ on the Riemann sphere $\mathbb{P}^{1}$. Then $\gamma=\sum_{i}\gamma_{i}$ defines a $(2n-2p+1)-$chain which is 
homologically trivial \cite{gl}.
\\
Let us define the $K$-theoretic $p$-th intermediate Jacobian of $X$ as 
\[
KJ^{p}(X)=\frac{\displaystyle H^{2p-2}(X,\mathbb{C})}{\displaystyle F^{p}H^{2p-2}(X,\mathbb{C})+H^{2p-2}(X,\mathbb{Z}(p))}\cong
\frac{\displaystyle (F^{n-p+1}H^{2n-2p+2}(X,\mathbb{C}))^{*}}{\displaystyle H_{2n-2p-2}(X,\mathbb{Z}(n-p+1))}.
\]
The regulator map $R_{p}:CH^{p}(X,1)\rightarrow KJ^{p}(X)$ is defined as
\[
R_{p}(Z)=\frac{\displaystyle1}{(2\pi i)^{n-p+1}}\Big(\sum_{i}\int_{Z_{i}}log(f_{i})+2\pi i\int_{D}\Big),
\] 
where $D$ is an $(2n-2p+2)-$chain such that $\partial D=\gamma$. 
\\
Let us define the Hodge group $Hdg^{p-1}(X):=H^{2p-2}(X,\mathbb{Z})\cap H^{p-1,p-1}(X)$.
Recall the following fact \cite{gl}.
\begin{lem} 
\label{indec}
If $R_{p}(Z)(\omega)\neq0$ for some $\omega$ orthogonal to the Hodge group $Hdg^{p-1}(X)$ then $Z$ is indecomposable.
\end{lem}
As in the case of Abel-Jacobi maps, the above construction can be performed in families.
To a family of $n$-dimensional smooth projective varieties $\pi:\mathcal{X}\rightarrow B$ we can associate the family of $K$-theoretic $p$-th Jacobians
\[
\mathcal{KJ}^{p}(\mathcal{X})=\frac{\displaystyle\mathcal{H}^{2p-2}}{\displaystyle\mathcal{F}^{p}\mathcal{H}^{2p-2}+R^{2p-2}\pi_{*}\mathbb{Z}(p)}
.\]
 If $\mathcal{Z}$ is an higher cycle on $\mathcal{X}$ such that its restriction to $X_{t}$ gives an higher 
cycle $Z_{t}\in CH^{p}(X_{t},1)$ we can define a section $\tau\in H^{0}(B,\mathcal{KJ}^{p}(\mathcal{X}))$ by the rule $\tau(t)=R_{p}(Z_{t})$. 

The section $\tau$ shares the formal properties of normal functions: it is holomorphic and, for every holomorphic local lifting $\widetilde{\tau}$ we have 
$\nabla\widetilde{\tau}\in$\mbox{$\mathcal{F}^{p-1}\mathcal{H}^{2p-2}\otimes\Omega^{1}_{B}$} (cf. \cite{v}). 
These properties allows to define the infinitesimal invariant $\delta\tau$ 
in a similar way to the classical case.
\subsection{Collino cycle and symmetric products}
\label{col}
We recall first the construction due to A. Collino \cite{c97} of a cycle in the higher Chow group of hyperelliptic Jacobians. 
We will use the following notation: if $D$ is a divisor on a curve we will denote by $[D]$ its class modulo linear equivalence.
Let $C$ be an hyperelliptic curve of genus $g>1$ and $h:C\rightarrow \mathbb{P}^{1}$ its covering map. We fix two ramification points 
$p,q\in C$ and a parameter on $\mathbb{P}^{1}$ such that $h(p)=0$, \mbox{$h(q)=\infty$.}
Note that $2[p-q]=0$ in $\Pic^{0}(C)$.
We can embed $C$ into $\Pic^{1}(C)$ and consider \mbox{$C(-p)=\{[x-p] \ | \ x\in C\}\subset\Pic^{0}(C)$} and $C(-q)$ defined in a similar way. 
Call $h_{p}$ and $h_{q}$ the functions on $C(-p),C(-q)\cong C$ induced by $h$. 
Then $\textrm{div}(h_{p})=2(0-[q-p])=2(0-[p-q])$ and 
$\textrm{div}(h_{q})=2([p-q]-0)$.
\\
The cycle $Z=Z(C,p,q)=(C(-p),h_{p})+(C(-q),h_{q})\in CH^{g}(J(C),1)$ is the cycle defined and studied by A. Collino \cite{c97}. He proved that 
$Z$ is indecomposable by computing explicitly the infinitesimal invariant of the corresponding normal function.

Using symmetric products it is possible to construct other higher cycles $Z_{k}=Z_{k}(C,p,q)\in CH^{g-k+1}(J(C),1)$ for $k=2,\dots$ \mbox{$g-1$}. 
In fact let us consider the rational function on the $k-$fold symmetric product of $C$, \mbox{$r:C_{k}\dashrightarrow\mathbb{C}$} 
defined on the open subset $C_{k}\setminus\{D\ | \ D-q\textrm{ is effective}\}$ 
by $r(x_{1}+\cdots+x_{k})=\prod_{i=1}^{k}h(x_{i})$.
Using the fact that the map $u:C_{k}\to W_{k}\subset\Pic^{k}(C)$, 
$u(\sum_{i=1}^{k}x_{i})=[\sum_{i=1}^{k}x_{i}]$, has a rational inverse 
(defined on the open subset $W_{k}^{0}=\{[D]\in W_{k} \ | \ h^{0}(D)=1 \}$),
 we get a diagram
\begin{equation}
\xymatrix{
 C_{k}\ar@{-->}[r]^{r} & \mathbb{C}    \\
W_{k}\ar@{-->}[u]\ar@{-->}[ur]^{l} &
}
\end{equation}
where $l$ is the composition.
By the very definition of $l$ we have $\textrm{div}(l)=2(W_{k-1}(p)-W_{k-1}(q))$, where $W_{k-1}(p)$ and $W_{k-1}(q))$ are the images 
of $C_{k-1}+p$ and $C_{k-1}+q$ in $J(C)$.
Let us now consider the bijections $i_{p}:\Pic^{k}(C)\to\Pic^{k-1}(C)$, $i_{q}:\Pic^{k}(C)\to\Pic^{k-1}(C)$ defined by $i_{p}([D])=[D-p]$ and $i_{q}([D])=[D-q]$.
 Set $W_{k}(-p)=i_{p}(W_{k})$ and $W_{k}(-q)=i_{q}(W_{k})$.
 We have rational functions $l_{p}=l\circ i_{p}^{-1}:W_{k}(-p)\dashrightarrow\mathbb{C}$ and 
 $l_{q}=l\circ i_{q}^{-1}:W_{k}(-q)\dashrightarrow\mathbb{C}$ whose divisors are
 $\textrm{div}(l_{p})=$ \mbox{$2(W_{k-1}-W_{k-1}(p-q))$} and 
$\textrm{div}(l_{q})=2(W_{k-1}(p-q)-W_{k-1})$.
Then $Z_{k}=(W_{k}(-p),l_{p})+(W_{k}(-q),l_{q})\in CH^{g-k+1}(J(C),1)$.
\\
To simplify the notation we set $Z_{k}=(A_{k},a_{k})+(B_{k},b_{k})$ for $k=1,\dots,g-1$. 
We also omit the index when $k=1$, i.e. we set $Z_{1}=Z=(A,a)+(B,b)$. 
 \begin{teo}
If $J(C)$ is a generic hyperelliptic jacobian of dimension $3$ then $Z_{2}$ is indecomposable.
\end{teo}
\begin{proof}
We first present some computations that are valid for every genus $g$ and every $k\leqslant g-1$.
Let us consider $\pi:\mathcal{C}\rightarrow B$ be a family of $2-$pointed hyperelliptic curves with $\pi^{-1}(0)=(C,p,q)$ and 
$j(\pi):\mathcal{J}\rightarrow B$ be corresponding 
family of Jacobians. 
Then $B$ is an open subset of the moduli space 
of hyperelliptic $2-$ pointed curves of genus $g$ and $T_{B,0}$ parametrizes first order hyperelliptic deformations of $(C,p,q)$.
We may assume the existence of two holomorphic sections $p,q:B\rightarrow\mathcal{C}$ such that $p(0)=p$ and $q(0)=q$.
Since the deformation is hyperelliptic there is a morphism $H:\mathcal{C}\rightarrow\mathbb{P}^{1}$ such that 
$h_{t}=H_{|C_{t}}:C_{t}\rightarrow\mathbb{P}^{1}$ is the covering map of the hyperelliptic curve $C_{t}$.
We also have for every $k$ a relative cycle $\mathcal{Z}_{k}=(\mathcal{A}_{k},F_{k})+
(\mathcal{B}_{k},G_{k})$ which restricts to $Z_{k}(t)=Z_{k}(C_{t},p(t),q(t))$ on $J(C)$. Here $F_{k},G_{k}$ are rational functons on 
$\mathcal{A}_{k}$ and $\mathcal{B}_{k}$ respectively.
Let us define the normal functions $\tau_{k}\in H^{0}(B,\mathcal{JK}^{g-k+1}(\mathcal{J}))$ by $\tau_{k}(t)=R_{g-k+1}(Z_{k}(t))$.
\\
We will show that $\delta\tau_{k}(0)\neq0$; 
this implies that $\tau_{k}$ is not a torsion section.
We then show that $Z_{2}$ is indecomposable when $g=3$ using Lemma \ref{indec}.

Note the commutative diagram
\begin{equation}
\label{infinitesimal}
\xymatrix{
 H^{1,1}(J(C))\ar[rr]^{\hspace{.4cm}(\nabla^{g-2}_{g}(0))^{t}}\otimes T_{B,0}\ar[d]^{\theta^{k-1}\otimes\textit{Id}} & &
H^{0,2}(J(C))\ar[d]^{\theta^{k-1}} \\
H^{k,k}(J(C))\otimes T_{B,0}\ar[rr]^{\hspace{.5cm}(\nabla^{g-k-1}_{g-k+1}(0))^{t}} & &
H^{k-1,k+1}(J(C))
.}
\end{equation}
By definition  $\delta\tau_{1}(0)\in(\ker\nabla^{g-2}_{g}(0)^{t})^{*}$ and 
\mbox{$\delta\tau_{k}(0)\in(\ker\nabla^{g-k-1}_{g-k+1}(0)^{t})^{*}$}.
Let us fix a non-zero $\zeta\in T_{B,0}$ and consider the associated first order deformation $\mathcal{J}_{\Delta}\rightarrow\Delta$, where 
$\Delta=\textrm{Spec}\frac{\mathbb{C}[t]}{(t^{2})}$. By pull-back, we can also suppose that $B$ is a small disc having tangent
 direction $\zeta$ at $0$. In particular $T_{B,0}=\langle\frac{\partial}{\partial t}\rangle\cong\langle\zeta\rangle$.
  \\
Take $\omega\in H^{1,0}(J(C))$ such that $\zeta\cdot\omega=0$
  and $\mu\in\zeta\cdot H^{1,0}(J(C))$. 
 Then $\omega\wedge\mu\in\ker(\nabla^{g-2}_{g}(0)^{t})$ and $\theta^{k-1}(\omega\wedge\mu)\in\ker(\nabla^{g-k-1}_{g-k+1}(0)^{t})$.
We now follow the argument of Collino. 
We denote by $\mu_{\Delta}$ the pullback to $\mathcal{J}_{\Delta}$ of the trivial lift of $\mu$ 
to $J(C)\times B$ (recall that this is diffeomorphic to $\mathcal{J}$).
 It is proved in \cite[pp. 405-406]{c97} that 
\begin{equation}
\label{tau1}
\delta\tau_{1}(0)(\omega\wedge\mu\otimes\zeta)=\int_{A}\mu_{\Delta}\wedge L_{\zeta}(\alpha_{\Delta}\wedge dF/F)+
\int_{B}\mu_{\Delta}\wedge L_{\zeta}(\alpha_{\Delta}\wedge dG/G).
\end{equation}
In the formula above $\alpha$ is a $C^{\infty}$ closed 1-form on $\mathcal{J}$ such that (i) its pullback $\alpha_{\Delta}$ to $\mathcal{J}_{\Delta}$ is holomorphic, 
(ii) $\alpha$ and the trivial lift of $\omega$ to $J(C)\times B$ coincides when pulled-back to $\mathcal{J}_{\Delta}$; 
$L_{\zeta}(\alpha_{\Delta}\wedge dF/F)=i(\frac{\partial}{\partial t})(\alpha_{\Delta}\wedge dF/F)$ is the Lie-contraction. 
The same computation of Collino gives
\begin{equation}
\label{tauk}
(2\pi i)^{k}\delta\tau_{k}(0)(\theta^{k-1}(\omega\wedge\mu)\otimes\zeta)=\int_{A_{k}}\theta^{k-1}\mu_{\Delta}\wedge L_{\zeta}(\alpha_{\Delta}\wedge dF_{k}/F_{k})+
\int_{B_{k}}\theta^{k-1}\mu_{\Delta}\wedge L_{\zeta}(\alpha_{\Delta}\wedge dG_{k}/G_{k}).
\end{equation}
Recalling that $F_{k},G_{k}$ are essentially the products of $F$ and $G$ ($k$ times), we find
\begin{equation}
\label{ }
(2\pi i)^{k}\delta\tau_{k}(0)(\theta^{k-1}(\omega\wedge\mu)\otimes\zeta)=k\int_{A_{k}}\theta^{k-1}\mu_{\Delta}\wedge L_{\zeta}(\alpha_{\Delta}\wedge dF/F)+
k\int_{B_{k}}\theta^{k-1}\mu_{\Delta}\wedge L_{\zeta}(\alpha_{\Delta}\wedge dG/G)
.
\end{equation} 
The term $L_{\zeta}(\alpha_{\Delta}\wedge dF/F)$ is by definition the adjoint form $\omega_{\zeta,\omega,dh/h}\in H^{0}(\omega_{C}(p+q))$
 and the same holds for $L_{\zeta}(\alpha_{\Delta}\wedge dG/G)$. 
Denoting by $\omega_{\zeta,\omega}$ its holomorphic part under the isomorphism $H^{0}(\omega_{C}(p+q))\cong H^{0}(\omega_{C})\oplus\langle dh/h\rangle$ 
we find (cf. \cite[p. 406]{c97})
\begin{equation}
\label{ }
\delta\tau_{1}(0)(\omega\wedge\mu\otimes\zeta)=2\int_{C}\mu\wedge\omega_{\zeta,\omega}=
\frac{\displaystyle 2}{\displaystyle (g-1)!}\int_{J(C)}\theta^{g-1}(\mu\wedge\omega_{\zeta,\omega})
\end{equation}
and 
\[
\delta\tau_{k}(0)(\theta^{k-1}(\omega\wedge\mu)\otimes\zeta)=
\frac{\displaystyle 2k}{\displaystyle (2\pi i)^{k}(g-k)!}\int_{J(C)}\theta^{g-1}(\mu\wedge\omega_{\zeta,\omega})= 
\frac{\displaystyle 2k(g-1)!}{\displaystyle (2\pi i)^{k}(g-k)!}\delta\tau_{1}(0)(\omega\wedge\mu\otimes\zeta).
\]
Collino proves that is possible to choose $\zeta,\omega,\mu$ such that $\omega\wedge\mu$ and $\delta\tau_{1}(\omega\wedge\mu\otimes\zeta)\neq0$.
Then for the same choices $\delta\tau_{k}(\theta^{k-1}(\omega\wedge\mu)\otimes\zeta)\neq0$.
\\
Let us now suppose that we are in the hypothesis of the Theorem.
Then since $J(C)$ is generic the Neron-Severi group, or equivalently the Hodge group of divisors, is generated by $\theta$ \cite{pir}.
Also the Hodge classes in codimension two  of any abelian threefold are intersection of those in codimension one (cf. \cite[Introduction]{mz}).  
We conclude that $Hdg^{2}(J(C))$ is generated by $\theta^{2}$ and then $\theta(\omega\wedge\mu)$ lies in its orthogonal space. 
The result follows from Lemma \ref{indec}.
\end{proof}

\textbf{Aknowledgements} I would like to thank G.P. Pirola for useful conversations on the subject and for his help with the results of Section
\ref{tecnico}. 
\\
I also thank A. Collino for letting me reproduce the proof of Lemma \ref{collino} due to him and for having pointed me out a mistake in the first
version of the paper.
\\  
Finally, I wish to thank Ramesh Sreekantan for a useful conversation on higher Chow cycles and for the correspondence we had.

\bigskip

\noindent Emanuele\ Raviolo\\
Dipartimento di Matematica, Universit\`a di Pavia \\
via Ferrata 1 \\
 27100 Pavia - Italy \\
 emanuele.raviolo@unipv.it

\end{document}